\documentclass[smallextended,referee,envcountsect,]{svjour3}
\smartqed
\usepackage{graphicx}
\usepackage{mathptmx}
\usepackage{latexsym}
\usepackage{amsfonts}
\usepackage{mathrsfs}

\usepackage{amsmath}
\usepackage{amssymb}
\journalname{JOTA}

\newcommand{\Pb}{\mathbb P}
\newcommand{\R}{\mathbb R}
\newcommand{\N}{\mathbb N}

\newcommand{\X}{\mathbb X}
\newcommand{\Y}{\mathbb Y}
\newcommand{\Uball}{{\mathbb B}}
\newcommand{\Usfer}{{\mathbb S}}

\newcommand{\dom}{{\rm dom}\, }
\newcommand{\graph}{{\rm gph}\,}

\newcommand{\nullv}{\mathbf{0}}

\newcommand{\bd}{{\rm bd}\, }
\newcommand{\inte}{{\rm int}\, }

\newcommand{\SFP}{{\rm SFP}}    

\newcommand{\MPEC}{{\rm MPEC}}    

\newcommand{\Solv}{\Sigma}    

\newcommand{\DRC}{{\rm DRC}}    

\newcommand{\GDRC}{{\rm GDRC}}    

\newcommand{\mf}{\nu}    

\newcommand{\Lin}{{\mathcal L}}

\newcommand{\stsl}[1]{|\nabla #1|}
\newcommand{\pastsl}[1]{|\nabla_x #1|}       

\newcommand{\ball}[2]{{\rm B}\left(#1;#2\right)}      
\newcommand{\cball}[2]{{\rm B}\left[#1;#2\right]}     

\newcommand{\lip}[2]{{\rm lip}(#1;#2)}     
\newcommand{\clm}[2]{{\rm clm}(#1;#2)}      
\newcommand{\dist}[2]{{\rm dist}\left(#1;#2\right)}
\newcommand{\exc}[2]{{\rm exc}(#1;#2)}

\newcommand{\Ncone}[2]{{\rm N}(#1;#2)}       
\newcommand{\haus}[2]{{\rm haus}(#1;#2)}

\newcommand{\Proj}[2]{\prod(#1;#2)}

\newcommand{\Lipusc}[2]{{\rm Lipusc}(#1;#2)}      

\newcommand{\Liplsc}[3]{{\rm Liplsc}(#1;#2,#3)}      
\newcommand{\clmsv}[3]{{\rm clm}(#1;#2,#3)}
\newcommand{\Lip}[3]{{\rm Lip}(#1;#2,#3)}      

\begin{document}

\title{On Lipschitzian properties of multifunctions defined
implicitly by ``split" feasibility problems}

\titlerunning{Lipschitzian properties of multifunctions defined
by ``split" feasibility problems}        


\author{Amos Uderzo}

\institute{Amos Uderzo \at
             Università di Milano-Bicocca \\
              Milan, MI 20125, Italy\\
              amos.uderzo@unimib.it
}

\date{Received: date / Accepted: date}

\maketitle

{\hfill Version updated: \today}.

\vskip.5cm

\begin{abstract}
In the present paper, a systematic study is made of quantitative
semicontinuity (a.k.a. Lipschitzian) properties of certain multifunctions, which
are defined as a solution map associated to a family of parameterized
``split" feasibility problems. The latter are a particular class
of convex feasibility problems with well recognized
applications to several areas of engineering and systems biology.
As a part of a perturbation analysis of variational systems,
this study falls within the framework of a line of research pursued by
several authors. It is performed by means of techniques of variational analysis,
which lead to establish sufficient conditions
for the Lipschitz lower semicontinuity, calmness,
isolated calmness, Lipschitz upper semicontinuity and Aubin property
of the solution map. Along with each of these properties, a quantitative estimate of the
related exact bound is also provided.
The key elements emerging on the way to achieving the main results
are dual regularity conditions qualifying the problem behaviour, which are
expressed in terms of convex analysis constructions involving
problem data.
The approach here proposed tries to unify the study of the
aforementioned properties.
\end{abstract}



\section{Introduction}

Let $(\X,\|\cdot\|)$ and $(\Y,\|\cdot\|)$ be real normed (vector) spaces.
Given a linear bounded operator $A:\X\longrightarrow\Y$, and two nonempty,
closed and convex sets $C\subseteq\X$ and $Q\subset\Y$, the following problem
$$
  \hbox{find } x\in C:\ A(x)\in Q.   \leqno (\SFP)
$$
is an abstract formulation of the ``split" feasibility problem, which was
originally introduced in \cite[Section 6.1]{CenElf94} in a finite-dimensional setting.
Since its solution set can be clearly represented as $\Sigma=C\cap A^{-1}(Q)$,
one sees that (\SFP) falls in the class of the convex feasibility problems, which
consist in finding a common point to several given convex sets
(see \cite{BauBor93,BauBor96,Breg67,Comb93,GuPoRa67}).
Nevertheless, for several reasons the $(\SFP)$ format turned out to be attractive
in itself, becoming soon a distinct research topic.
Its ``split" structure, connecting different spaces, can trigger
certain advantages from both the computational and the applicative viewpoint.
The main solution algorithms (e.g. CQ algorithms) are known to be based on
projections. Projecting onto $A^{-1}(Q)$, provided that such set can be
actually computed, is typically harder than projecting onto $C$ or $Q$.
Besides, $C$ and $Q$ may happen to live in Euclidean spaces with highly
different dimensions. In the context of applications, the ``split" structure
can reflect more faithfully the physical reality of a model.
Having found relevant applications in signal processing and in image
reconstruction (see \cite{CeElKoBo05,Hurt89}), (\SFP) has been widely studied,
especially from the point of view of solution algorithms
(see \cite{CenElf94,HuXuYe25,HuXuYe25b,ReTrMa20,ReiTuy21,WaHuLiYa17,WaGaLiYa24}).

In the present paper, a parameterized version of ``split" feasibility problems
will be considered in the following terms. Let $(P,d)$ be
a metric space representing the set of all parameter perturbations.
Given a mapping $A:P\times\X\longrightarrow\Y$ and two
set-valued maps $C:P\rightrightarrows\X$ and $Q:P\rightrightarrows\Y$,
any element $p\in P$ determines the ``split" feasibility problem
$$
  \hbox{find } x\in C(p):\ A(p,x)\in Q(p).   \leqno (\SFP_p)
$$
The solution set-valued map $\Solv:P\rightrightarrows\X$ associated to the resulting family of perturbed
``split" feasibility problems $(\SFP_p)$ is therefore given by
\begin{equation}     \label{eq:Solvpsetchar}
  \Solv(p)=\left\{x\in C(p):\ A(p,x)\in Q(p)\right\}=
  C(p)\cap A(p,\cdot)^{-1}(Q(p)).
\end{equation}
The subject of this paper are so-called Lipschitzian
properties of the multifunction $\Solv$.
The study here exposed follows a research line recently
initiated by \cite{HuXuYe25,HuXuYe25b} and should be regarded
as a part of a more comprehensive perturbation analysis of
$(\SFP)$. This term denotes an apparatus of systematic investigations
about effects produced on the solution set by changes of problem data,
which may be caused by different reasons. These investigations
may be qualitative, e.g. when dealing with existence of solutions (solvability)
or their topological properties under changes of data, or quantitative,
when focusing on the measurement of variations of the solution set
as a consequence of measured changes of the perturbation parameter
(sensitivity).
The perturbation analysis is a crucial issue in understanding the
deep nature of many mathematical problems. In the specific case of
$(\SFP)$, some reasons for conducting a perturbation analysis are
well explained in \cite{HuXuYe25,HuXuYe25b}. To them, some further
motivations can be added, which come from the topics described below:

\begin{itemize}
  \item[$\bullet$] {\it mathematical programming with $(\SFP)$ constraints}.
  Given $\vartheta:P\times\X\longrightarrow\R\cup\{\pm\infty\}$,
  $\Omega\subseteq P$ and a family $(\SFP_p)$, consider the optimization
  problem
\begin{equation}    \label{def:mpsfp}
  \min \vartheta(p,x)\ \hbox{ subject to }\ p\in\Omega,\ x\in\Solv(p),
\end{equation}
where $\Solv$ is as in (\ref{eq:Solvpsetchar}).
The definition of the feasible region of such an optimization problem
involves lower level ``split" feasibility problems to be solved for each
value of $p$. This may be not always possible or convenient for the
analysis of the given optimization problem, so the behaviour of $\Solv$
as an implicit multifunction becomes an important issue.
Indeed, the problem (\ref{def:mpsfp}) can be reformulated as a mathematical
program with equilibrium constraints (for short, $\MPEC$). A successful approach to
derive optimality conditions for $\MPEC$ is known to rely on the
Lipschitzian stability of parameterized variational systems
(see \cite[Section 5.2]{Mord06b}).

\vskip.25cm

  \item[$\bullet$] {\it bi-level $(\SFP)$}. Given two maps
  defined on normed spaces $A_1,\, A_2:\Pb\times\X\longrightarrow\Y$,
  with $A_1(\cdot,x)\in\Lin(\Pb,\Y)$ for each $x\in\X$ and
  $A_2(p,\cdot)\in\Lin(\X,\Y)$ for each $p\in\Pb$,
  two subsets $C_1\subseteq\Pb$ and $Q_1\subseteq\Y$, and two set-valued maps
  $C_2:\Pb\rightrightarrows\X$ and $Q_2:\Pb\rightrightarrows\Y$, the problem
  $$
  \hbox{find } p\in C_1:\ A_1(p,x)\in Q_1 \hbox{ and }  x\in\Solv(p),
  $$
where $\Solv:\Pb\rightrightarrows\X$ is the solution map of the lower
level problem
$$
  \hbox{find } x\in C_2(p):\ A_2(p,x)\in Q_2(p),
$$
can be referred to as a bi-level ``split" feasibility problem. Also in this context,
without solving each lower level problem, one can
take profit from indirect information about the behaviour of the
multifunction $\Solv$.

\vskip.25cm

  \item[$\bullet$] {\it approximated solution methods}. It is well known
  that there are some classes of $(\SFP)$ significantly more tractable,
  from both the theoretical and the computational viewpoint, than others.
  An instance of this is provided by $(\SFP)$ with polyhedral data:
  advantages in computing projections (on which CQ algorithms rely) and
  the possibility of exploiting linear algebra tools make them preferable
  to $(\SFP)$ with general convex data. Thus, one could consider to perturb
  a given $(\SFP)$ in such a way to generate a sequence $(\SFP_{p_n})_{n\in\N}$,
  with each problem of the sequence (approximated problem) having polyhedral data.
  Such an approach opens the question whether sequences of solutions to
  approximated problems do and, if so, how converge to a solution of the
  original $(\SFP)$. Meaningful insights into this question can be afforded
  from a knowledge of the behaviour of the multifunction $\Solv$.
\end{itemize}

Within the wider context of convex feasibility problems, situations
where data vary have been studied also in \cite{CeGiLeSc16,GhTeStCe13},
in modeling cooperative wireless sensor network positioning.
Nevertheless, in those problems a parametrization of data, which is
independent of the problem unknown, is not considered.
Consequently, a study of the solution set as a multifunction
depending on the parameter variables can not be found there.

\subsection{Outline, methodology and contributions}

The contents of the paper are arranged as follows.
In the next subsection the basic notations employed in the paper
are listed.
In Section \ref{Sect:2}, the definition of the main properties under investigation
is recalled, relationships between some of them are pointed out, and needed technical
preliminaries are gathered.
In Section \ref{Sect:3}, the main results are exposed. They are sufficient conditions
for $\Solv$ to satisfy each of the Lipschitzian properties recalled
in Section \ref{Sect:2}, with the perturbation parameter $p$ varying in
a metric space. Each of them is complemented  with an estimate of the related
exact bound, quantifying ``how much" the property is fulfilled.
In Section \ref{Sect:4}, a specialized sufficient condition for the Aubin
property of $\Solv$ is established under a convexity assumption, with the
perturbation parameter now varying in a normed space.

A feature of the present study to be remarked is that all of the presented
sufficient conditions, except one, are derived from the same preliminary
result, which states a condition for local solvability and error bounds.
Such a feature reflects the methodology chosen to accomplish the task
of the present study. Instead of performing one independent analysis for
each of the Lipschitzian properties, according to the approach here adopted
all the results are achieved as a consequence of a unique qualification condition,
and its nonlocal counterpart in the case of upper Lipschitz continuity.
The price to be paid for getting a unifying scheme of analysis is that, in general, such
qualification condition turns out to be too strong to obtain full
characterizations.

In fact, $(\SFP_p)$ can be regarded as a class of parametric variational systems,
namely a parameterized version of generalized equations. A well developed
literature devoted to the study of the Lipschitzian properties
of the related solution map is nowadays at disposal within variational
analysis (see, among others, \cite{DonRoc14,LevRoc94,Mord06,Robi79,Robi80}).
A characterization of the Aubin property of $\Solv$ has been established
in \cite{HuXuYe25}, in the particular case of a finite-dimensional setting,
in which $C$ and $Q$ are supposed to be constant (given sets) and the map $A$
is perturbed in the Euclidean space of all the matrices with real entries.
This result was achieved by exploiting a finite-dimensional characterization
of the Aubin property of solution map to variational systems due to B.S.
Mordukhovich, which relies on the coderivative calculus.
In contrast, to perform a comprehensive analysis allowing perturbations
to vary in a metric space, a different approach has been followed in the
present paper. Whereas some constructions based on subdifferential calculus
are still used, solvability and a key error bound estimate, on which the main
results are based, are obtained by means of elementary tools of variational analysis
in metric spaces (partial strong slope), via a scalarization method
in the spirit of \cite{BorZhu05}.

\subsection{Notations}

The basic notations in use throughout the paper are standard.
$\R$ denotes the field of real numbers, $\N$ denotes its subset of
the natural ones, while $\R^d$ stands for the $d$-dimensional Euclidean space.
Given a (extended) scalar function $\varphi:X\longrightarrow\R\cup\{\pm\infty\}$
defined on a metric space $X$,
$\dom\varphi=\varphi^{-1}(\R)$ denotes its domain and, if $\alpha\in\R$,
$[\varphi\le\alpha]=\{x\in X\ :\ \varphi(x)\le\alpha\}$ denotes its $\alpha$-sublevel
set, $[\varphi>\alpha]=\{x\in X:\ \varphi(x)>\alpha\}$
denotes its $\alpha$-strict superlevel set, whereas $[\varphi=\alpha]=\varphi^{-1}(\alpha)$
the $\alpha$-level set.
If $\Omega$ is a subset of $X$,
$\dist{x}{\Omega}=\inf_{z\in\Omega}d(z,x)$ denotes that distance of a point $x\in X$ from
$\Omega$, with the convention that $\dist{x}{\varnothing}=+\infty$.
Consistently, if $r\ge 0$, $\cball{\Omega}{r}=[\dist{\cdot}{\Omega}\le r]$
indicates the closed $r$-enlargement of the set $\Omega$, with radius $r$. In particular, if
$\Omega=\{x\}$, $\cball{x}{r}$ denotes the closed ball with center $x$
and radius $r$.
Similarly, if $r>0$, $\ball{\Omega}{r}=[\dist{\cdot}{\Omega}<r]$ and
$\ball{x}{r}$ stand for the open $r$-enlargement of $\Omega$ and the open ball centered
at $x$, with radius $r$.
The symbol $\bd\Omega$ and $\inte\Omega$ indicate the boundary and the topological interior
of $\Omega$, respectively.
Given a pair of subsets of the same space, say $\Omega$ and $S$, the excess of $\Omega$ over $S$
is indicated by $\exc{\Omega}{S}=\sup_{x\in \Omega}\dist{x}{S}$, whereas their
Hausdorff-Pompeiu distance by $\haus{\Omega}{S}=\max\{\exc{\Omega}{S},\, \exc{S}{\Omega}\}$.
Given a set-valued map $F:P\rightrightarrows X$ between metric spaces, $\dom F=
\{p\in P:\ F(p)\ne\varnothing\}$ denotes the effective domain
of $F$, while $\graph F=\{(p,x)\in P\times X:\ x\in F(p)\}$ its graph.

Whenever $(\X,\|\cdot\|)$ is a normed space, $\nullv$ stands for its null element.
The (topological) dual space of $\X$ is denoted by $\X^*$, with $\langle\cdot,\cdot\rangle
:\X^*\times\X\longrightarrow\R$ denoting their duality pairing and $\nullv^*$
the zero functional over $\X$. In this context, $\cball{\nullv}{1}$ and $\bd\cball{\nullv}{1}$
(resp. $\cball{\nullv^*}{1}$ and $\bd\cball{\nullv^*}{1}$)
will be simply indicated by $\Uball$ and $\Usfer$ (resp. $\Uball^*$ and $\Usfer^*$ ), respectively.
By $\Lin(\X,\Y)$ the space of all bounded linear operators between $\X$ and $\Y$
is denoted. This space will be equipped with the operator norm, indicated by $\|\cdot\|_{\Lin}$.
Given $A\in\Lin(\X,\Y)$, $A^*:\Y^*\longrightarrow\X^*$ stands for the adjoint
operator to $A$.

The acronyms l.s.c. and u.s.c., standing for lower semicontinuous
and upper semicontinuous, respectively, will be used.
The meaning of some further symbols employed in the subsequent sections
will be explained contextually to their introduction.

\vskip.5cm


\section{Preliminaries}     \label{Sect:2}

Throughout the paper, with reference to a family $(\SFP_p)$, the following assumptions on the problem data
are maintained:

\begin{itemize}

\item[($a_0$)] $(\X,\|\cdot\|)$ is a Banach space;

\item[($a_1$)] the set-valued maps $C:P\rightrightarrows\X$ and $Q:P\rightrightarrows\Y$ take
nonempty, closed and convex values;

\item[($a_2$)] $A(p,\cdot)\in\Lin(\X,\Y)$ for every $p\in P$.

\end{itemize}

The assumption ($a_0$) enables to derive solution existence from the metric
completeness of the decision space, via its variational reformulation. Therefore,
as long as following the present approach, it seems to be hardly avoidable.
The role of assumptions ($a_1$) and ($a_2$) is to ensure that, even in
the presence of perturbations, the data of each instance of $(\SFP_p)$
match the assumptions of $(\SFP)$.

\begin{remark}    \label{rem:Solvtoppro}
From the equalities in (\ref{eq:Solvpsetchar}) it should be clear
that under assumptions ($a_1$) and ($a_2$),
as an intersection of closed convex sets, the set $\Solv(p)$ is
(possibly empty) closed and convex for every $p\in P$. In other words,
the multifunction $\Solv:P\rightrightarrows\X$ always takes closed and
convex values.
Since $\X$ is, in particular, a metric space, it is regular as a topological
space. Therefore, according to \cite[Theorem 17.10]{AliBor06}, whenever
$\Solv$ is u.s.c., $\graph\Solv$ turns out to be a closed subset of $P\times\X$.
\end{remark}


\subsection{Solution map scalarization}

A scalar characterization of the multifunction $\Solv$ can be
performed via the following merit function: let
$\mf:P\times\X\longrightarrow [0,+\infty)$ be defined by
\begin{equation}    \label{eq:mfunctdef}
    \mf(p,x)=\dist{A(p,x)}{Q(p)}+\dist{x}{C(p)}.
\end{equation}
It is worthwhile remarking that by virtue of assumption $(a_1)$, $\dom\mf=
P\times\X$. The relation connecting $\Solv$ and $\mf$ is made
clear below.

\begin{proposition}      \label{pro:scacharSolv}
Given a family $(\SFP_p)$, it holds:
\begin{equation}   \label{eq:Solvchar}
  \Solv(p)=[\mf(p,\cdot)\le 0]=[\mf(p,\cdot)=0].
\end{equation}
\end{proposition}
\begin{proof}
Both the equalities in the thesis follow at once from the nonnegativity
of $\mf$ and the closedness of the values of $C$ and $Q$, according
to assumption $(a_1)$.
\smartqed\qed
\end{proof}

It should be evident that definition in (\ref{eq:mfunctdef}) is not
the unique possible choice for a merit function associated to $(\SFP_p)$.
Within a scalarization approach,
parallel studies of the topic at the issue could be conducted
by means of any function of the form $\mf_{\|\cdot\|}$, given by
$$
  \mf_{\|\cdot\|}(p,x)=\|(\dist{A(p,x)}{Q(p)}\, ,\,\dist{x}{C(p)})\|,
$$
where $\|\cdot\|$ is any norm of $\R^2$.


\subsection{Tools of variational analysis in metric spaces}     \label{Sect:1}

The main properties, investigated in the present paper with
reference to the solution map $\Solv$, are recalled in the
next definition in a general setting.

\begin{definition}[Lipschitzian properties]    \label{def:Lipsemcont}
Let $F:P\rightrightarrows X$ be a set-valued mapping between
metric spaces and let $(p_0,x_0)\in\graph F$. The map $F$ is said:

\begin{itemize}

\item[(i)] to be {\it Lipschitz lower semicontinuous} at $(p_0,x_0)$ if
there exist positive constants $\delta$ and $\ell$ such that
\begin{equation}     \label{in:defLiplsc}
  F(p)\cap\cball{x_0}{\ell d(p,p_0)}\ne\varnothing,\quad
  \forall p\in\cball{p_0}{\delta};
\end{equation}
the value
$$
  \Liplsc{F}{p_0}{x_0}=\inf\{\ell>0:\ \exists\delta>0
  \hbox{ for which (\ref{in:defLiplsc}) holds}\}
$$
is called {\it exact Lipschitz lower semicontinuity bound} of $F$
at $(p_0,x_0)$.

\item[(ii)] to be {\it calm} at $(p_0,x_0)$ if there exist
positive constants $\delta$, $\zeta$ and $\ell$ such that
\begin{equation}    \label{in:defsvcalm}
   F(p)\cap\cball{x_0}{\zeta}\subseteq\cball{F(p_0)}{\ell
   d(p,p_0)},\quad\forall p\in\cball{p_0}{\delta};
\end{equation}
the value
$$
  \clmsv{F}{p_0}{x_0}=\inf\{\ell>0:\ \exists\delta,\,\zeta>0
  \hbox{ for which (\ref{in:defsvcalm}) holds}\}
$$
is called {\it exact calmness bound} of $F$ at $(p_0,x_0)$.

\item[(iii)] to be {\it Lipschitz upper semicontinuous} (or
{\it outer Lipschitz continuous/upper Lipschitzian}) at $p_0$ if
there exist positive constants $\delta$ and $\ell$ such that
\begin{equation}    \label{in:defLipusc}
   F(p)\subseteq\cball{F(p_0)}{\ell
   d(p,p_0)},\quad\forall p\in\cball{p_0}{\delta};
\end{equation}
the value
$$
  \Lipusc{F}{p_0}=\inf\{\ell>0:\ \exists\delta>0
  \hbox{ for which (\ref{in:defLipusc}) holds}\}
$$
is called {\it exact Lipschitz upper semicontinuity bound} of $F$
at $p_0$.

\item[(iv)] to satisfy the {\it Aubin property} (or, equivalently,
to be {\it Lipschitz-like}/{\it pseudo-Lipschitz}) around $(p_0,x_0)$ if there exist
positive constants $\delta$, $\eta$ and $\ell$ such that
\begin{equation}    \label{in:defAubin}
   F(p_1)\cap\cball{x_0}{\eta}\subseteq\cball{F(p_2)}{\ell
   d(p_1,p_2)},\quad\forall p_1,\, p_2\in\cball{p_0}{\delta};
\end{equation}
the value
$$
  \Lip{F}{p_0}{x_0}=\inf\{\ell>0:\ \exists\delta,\,\eta>0
  \hbox{ for which (\ref{in:defAubin}) holds}\}
$$
is called {\it exact Aubin property bound} of $F$ around $(p_0,x_0)$.
\end{itemize}
\end{definition}

The existence of a varying terminology and of several
equivalent reformulations of the above properties (e.g. in terms of
semiregularity, metric subregularity, metric regularity, referred to the
inverse multifunction, as one can see in \cite[Section 5.5.4]{BorZhu05},
\cite[Chapter 3]{DonRoc14}, \cite[Chapter 2]{Ioff17} \cite[Chapter 1]{KlaKum02},
\cite[Chapter 1]{Mord06}, \cite[Chapter 9]{RocWet98}) witness and reflect
the crucial role of the above properties in the development of modern
variational analysis.

For the reader's convenience, a scheme of well-known implications
between the properties recalled in Definition \ref{def:Lipsemcont}
is represented below, as it can be drawn from their very definition.
$$
\boxed{
\begin{array}{ccccc}
  \hbox{Lip. u.s.c. at } p_0 & \ \Rightarrow\ & \hbox{calmness at } (p_0,x_0) & \ \Leftarrow\ & \hbox{Aubin property at } (p_0,x_0) \\
   &  &  &  & \Downarrow \\
   &  &  &  & \hbox{Lip. l.s.c. at } (p_0,x_0)
\end{array}
}
$$
Counterexamples can be found (see, for instance, \cite{Uder14b}), which
justify the following negations of implication:
$$
\boxed{
\begin{array}{ccc}
  \hbox{Lip. l.s.c. at } (p_0,x_0) & \quad \begin{array}{c}
                                             \nRightarrow  \\
                                             \nLeftarrow
                                           \end{array}  \quad & \hbox{calmness at } (p_0,x_0),
\end{array}
}
$$
whence the independence of these two properties.

\begin{remark}     \label{rem:lscLiplsc}
(i) In view of a subsequent use, let us recall that if
$F:P\rightrightarrows X$ is Lipschitz l.s.c.
at $(p_0,x_0)\in\graph F$, then it is also l.s.c. in the sense
of \cite{Borw81} at the same point, namely for every $\epsilon>0$
there must exist $\delta_\epsilon>0$ such that
$$
  F(p)\cap\ball{x_0}{\epsilon}\ne\varnothing,\quad\forall
  p\in\cball{p_0}{\delta_\epsilon}.
$$
The latter is a localized counterpart of the classic
notion of lower semicontinuity for set-valued maps acting in
metric spaces.

(ii) Following \cite[Section 3.2]{DonRoc14},
a set-valued map $F:P\rightrightarrows X$ between metric spaces
is said to be inner semicontinuous (for short, i.s.c.) at $p_0\in P$
provided that $\displaystyle\lim_{p\to p_0}
\exc{F(p_0)}{F(p)}=0$. In connection with the present analysis,
it is useful to notice that if $F$ is i.s.c. at $p_0$, then it is
l.s.c. at each $(p_0,x_0)$, for every $x_0\in F(p_0)$.
Indeed, fixed $\epsilon>0$, by inner semicontinuity one gets the
existence of $r_\epsilon>0$ such that
$$
   \exc{F(p_0)}{F(p)}=\sup_{x\in F(p_0)}\dist{x}{F(p)}<\epsilon,
   \quad\forall p\in\cball{p_0}{r_\epsilon}.
$$
So, taken an arbitrary $x_0\in F(p_0)$, from the last inequality
one obtains in particular
$$
  \dist{x_0}{F(p)}<\epsilon,
  \quad\forall p\in\cball{p_0}{r_\epsilon},
$$
which implies
$$
  F(p)\cap\ball{x_0}{\epsilon}\ne\varnothing,\quad\forall
  p\in\cball{p_0}{r_\epsilon}.
$$
Moreover, it is known that $F$ is i.s.c. at $p_0$ iff
the function $p\mapsto\dist{x}{F(p)}$ is u.s.c. at $p_0$
for every $x\in X$ (see \cite[Theorem 3B.2(g)]{DonRoc14}.
So, if $x_0\in F(p_0)$,
this fact implies that if $F$ is i.s.c. at $p_0$, then the
the function $(p,x)\mapsto\dist{x}{F(p)}$ is u.s.c. at $(p_0,x_0)$.

(iii) The reader should notice that, whereas the properties in Definition
\ref{def:Lipsemcont}(i), (ii) and (iv) are referred to a point in the graph
of a given multifunction, Lipschitz upper semicontinuity is referred
to a point in the domain space $P$. Thus, the related implication in the above scheme
means that if $F$ is Lipschitz u.s.c. at $p_0$, then it is calm at each
point $(p_0,x_0)$, for every $x_0\in F(p_0)$.

(iv) Whenever $F:P\longrightarrow X$ is a single-valued map, $F$
is called calm at $p_0\in P$ if there exist positive constants $\delta$
and $\ell$ such that
\begin{equation}   \label{in:defcalmsinglevm}
  d(F(p),F(p_0))\le\ell d(p,p_0),\quad\forall
  p\in\cball{p_0}{\delta},
\end{equation}
with the value
$$
   \clm{F}{p_0}=\inf\{\ell>0:\ \exists\delta>0
  \hbox{ for which (\ref{in:defcalmsinglevm}) holds}\}
$$
providing the related bound. The reader should be aware of the fact
that such a definition is not consistent, in general, with Definition \ref{def:Lipsemcont}(ii),
in the sense that it does not coincide with the specialization of the
latter to the case of single-valued maps (see \cite[Remark 3.1]{Uder14}).

(v) Whenever $F$ is single-valued in a neighbourhood near $p_0$,
the Aubin property coincides with the classical Lipschitz continuity
near $p_0$, whose bound is denoted by $\lip{F}{p_0}$. In such an event,
one has
$$
  \Lip{F}{p_0}{F(p_0)}=\lip{F}{p_0}.
$$
\end{remark}

As an infinitesimal tool to be used in connection with the scalarization
approach, the (partial) strong slope with respect to
the variable $x$ of a function $\varphi:P\times X\longrightarrow\R
\cup\{\pm\infty\}$, defined on the product of metric spaces, at
$(p_0,x_0)\in P\times X$, will be considered, which is defined as
\begin{eqnarray*}
  \pastsl{\varphi}(p_0,x_0)=\left\{\begin{array}{ll}
  0, & \hbox{\ if $(p_0,x_0)$ is a local} \\
    & \hbox{\ minimizer of $\varphi$,}     \\
    & \\
  \displaystyle\limsup_{x\to x_0}{\varphi(p_0,x_0)-\varphi(p_0,x)\over d(x,x_0)}, &
  \hbox{\ otherwise.}
  \end{array}
  \right.
\end{eqnarray*}

The fundamental fact that will be employed for ensuring solution existence and for providing
useful error bounds via a strong slope condition is the following Basic Lemma.
For its proof in a metric space setting, which is based on the
Ekeland variational principle, the reader is referred to \cite[Lemma 2.3]{Uder21}.
Another (earlier) formulation in a more structured setting can be found in
\cite[Theorem 3.6.3]{BorZhu05}.

\begin{lemma}[Basic Lemma]     \label{lem:basicl}
Let $P$ and $X$ be metric spaces, let $(\bar p,\bar x)\in P\times X$
and let $\varphi:P\times X\longrightarrow [0,+\infty)$ be a given function.
Suppose that:
\begin{itemize}
  \item[(i)] $(X,d)$ is metrically complete;

  \item[(ii)] $\varphi(\bar p,\bar x)=0$;

  \item[(iii)] the function $p\mapsto\varphi(p,\bar x)$ is
  (u.s.) continuous at $\bar p$;

  \item[(iv)] $\exists r>0$ such that each function $x\mapsto\varphi(p,x)$
  is l.s.c. on $X$, for every $p\in\ball{\bar p}{r}$;

  \item [(v)] $\exists \delta>0$ such that
  $$
     \tau_\nabla(\delta)=\inf\{\pastsl{\varphi}(p,x):\ (p,x)\in [\ball{\bar p}{\delta}\times
     \ball{\bar x}{\delta}]\cap [\varphi>0]\}>0.
  $$
\end{itemize}
Then, there exist $\eta,\, \zeta>0$ such that:
\begin{itemize}
  \item[(t)] $[\varphi(p,\cdot)\le 0]\cap\cball{\bar x}{\eta}\ne\varnothing,
  \quad\forall p\in \cball{\bar p}{\zeta}$;

  \item[(tt)] the following estimate holds
\begin{equation}    \label{in:merbo}
  \dist{x}{[\varphi(p,\cdot)\le 0]}\le\tau_\nabla^{-1}(\delta)\varphi(p,x),\quad\forall(p,x)\in
  \cball{\bar p}{\zeta}\times \cball{\bar x}{\eta}.
\end{equation}
\end{itemize}

\end{lemma}

The next lemma provides a sufficient condition for the continuity of
the function $p\mapsto\mf(p,\bar x)$ defined as in (\ref{eq:mfunctdef}),
which will be employed in the next section.

\begin{lemma}     \label{lem:contmf}
With reference to a family $(\SFP_p)$, let $(\bar p,\bar x)\in\graph\Solv$.
Suppose that:
\begin{itemize}
  \item[(i)]  $C:P\rightrightarrows\X$ is l.s.c. at $(\bar p,\bar x)$;

  \item[(ii)] $Q:P\rightrightarrows\Y$ is l.s.c. at $\left(\bar p,A(\bar p,\bar x)\right)$;

  \item[(iii)] $A(\cdot,\bar x):P\rightarrow\Y$ is continuous at $\bar p$.
\end{itemize}

\noindent Then, the function $p\mapsto\mf(p,\bar x)$ is continuous at $\bar p$.
\end{lemma}

\begin{proof}
Fix an arbitrary $\epsilon>0$. According to hypothesis (i), corresponding to
$\epsilon/3$ there exists $\delta_{C,\epsilon}>0$ such that
$$
  C(p)\cap\ball{\bar x}{\epsilon/3}\ne\varnothing,\quad\forall
  p\in\cball{\bar p}{\delta_{C,\epsilon}},
$$
and hence
\begin{equation}   \label{in:lemcontmf1}
  \dist{\bar x}{C(p)}\le\frac{\epsilon}{3},\quad\forall
  p\in\cball{\bar p}{\delta_{C,\epsilon}}.
\end{equation}
Similarly, by hypothesis (ii), corresponding to
$\epsilon/3$ there exists $\delta_{Q,\epsilon}>0$ such that
$$
  Q(p)\cap\ball{A(\bar p,\bar x)}{\epsilon/3}\ne\varnothing,\quad\forall
  p\in\cball{\bar p}{\delta_{Q,\epsilon}},
$$
and hence
\begin{equation}   \label{in:lemcontmf2}
  \dist{A(\bar p,\bar x)}{C(p)}\le\frac{\epsilon}{3},\quad\forall
  p\in\cball{\bar p}{\delta_{Q,\epsilon}}.
\end{equation}
According to hypothesis (iii), corresponding to
$\epsilon/3$ there exists $\delta_{A,\epsilon}>0$ such that
\begin{equation}   \label{in:lemcontmf3}
  \|A(p,\bar x)-A(\bar p,\bar x)\|\le\frac{\epsilon}{3},\quad\forall
  p\in\cball{\bar p}{\delta_{A,\epsilon}}.
\end{equation}
Thus, by taking $\delta_\epsilon=\min\{\delta_{C,\epsilon},\, \delta_{Q,\epsilon},\,
\delta_{A,\epsilon}\}$, as a consequence of inequalities (\ref{in:lemcontmf1}),
(\ref{in:lemcontmf2}), and (\ref{in:lemcontmf3}), one obtains
\begin{eqnarray*}
  \mf(p,\bar x) &=& \dist{A(p,\bar x)}{Q(p)}+\dist{\bar x}{C(p)} \\
   &\le& \|A(p,\bar x)-A(\bar p,\bar x)\|+\dist{A(\bar p,\bar x)}{Q(p)}+\dist{\bar x}{C(p)}  \\
   &\le& \frac{\epsilon}{3}+\frac{\epsilon}{3}+\frac{\epsilon}{3}=\mf(\bar p,\bar x)+\epsilon,
   \quad\forall   p\in\cball{\bar p}{\delta_\epsilon}.
\end{eqnarray*}
The thesis follows from the last inequality, as $\mf$ is nonnegative
by its definition.
\smartqed\qed
\end{proof}

\begin{remark}    \label{rem:Lipcontmf}
For the purposes of the present analysis it is useful
to notice that, as a consequence of the (Lipschitz) continuity of the
distance function from a nonempty closed set and of the assumption $(a_2)$,
each function $x\mapsto\mf(p,x)$ is (Lipschitz) continuous on $\X$,
for every $p\in P$.
\end{remark}

As a further tool for investigating the Aubin property of $\Solv$,
the following scalarization criterion will be exploited.

\begin{proposition}    \label{pro:scacriAubcont}
For any multifunction $F:P\rightrightarrows\X$ and $(p_0,x_0)\in\graph F$,
the following properties are equivalent:
\begin{itemize}
  \item[(i)] $F$ has the Aubin property around $(p_0,x_0)$;

  \item[(ii)] the scalar function $(p,x)\mapsto\dist{x}{F(p)}$
  is locally Lipschitz around $(p_0,x_0)$.
\end{itemize}
\end{proposition}

This criterion was first established in a finite-dimensional setting in
\cite[Theorem 2.3]{Rock85}. Then, an extension to the case of multifunctions
acting between Banach spaces has been provided in \cite[Theorem 1.41]{Mord06}.
A perusal reveals that the linear structure as well as the completeness
of the involved spaces plays no role in the proof, so the same criterion
can be further extended to the setting considered in
Proposition \ref{pro:scacriAubcont}. Nevertheless, in Section \ref{Sect:4}
Proposition \ref{pro:scacriAubcont} will be applied with $P$ being a
normed space.


\subsection{Tools from convex analysis}

The regularity conditions, upon which the main results will be established,
are expressed in terms of some basic constructions of convex analysis,
dealing with the approximation of sets and functions. In what follows,
their definitions and the needed calculus rules are briefly recalled.
Given a convex set $\Omega\subseteq\X$ and $\bar x\in\Omega$, the set
$$
   \Ncone{\bar x}{\Omega}=\{x^*\in\X^*\ :\ \langle x^*,x-\bar x\rangle\le 0,
   \quad\forall x\in \Omega\}
$$
is the normal cone to $\Omega$ at $\bar x$. By its very definition, it is clear that
the normal cone is always a closed, convex cone in $\X^*$, containing $\nullv^*$. Moreover,
whenever $\bar x\in\inte \Omega$, one has $\Ncone{\bar x}{\Omega}=\{\nullv^*\}$.

Given a convex function $\varphi:\X\longrightarrow\R\cup\{\pm\infty\}$
and $\bar x\in\dom\varphi$, the (Moreau-Rockafellar) subdifferential of $\varphi$
at $\bar x$ is the set
$$
   \partial\varphi(\bar x)=\{x^*\in\X^*\ :\ \langle x^*,x-\bar x\rangle\le\varphi(x)
   -\varphi(\bar x),\quad\forall x\in\X\}.
$$
This derivative-like object is known to enjoy a rich apparatus of calculus.
Among the subdifferential calculus rules, the following ones will be employed in the sequel:

\begin{itemize}

\item[$\ $] sum rule: if $\varphi:\X\longrightarrow\R\cup\{\pm\infty\}$ and
$\vartheta:\X\longrightarrow\cup\{\pm\infty\}$ are convex functions, with at least one
of them being continuous at $\bar x\in\dom\varphi\cap\dom\vartheta$, then the following equality holds
\begin{equation}    \label{eq:subdsumr}
   \partial\left(\varphi+\vartheta\right)(\bar x)=
   \partial\varphi(\bar x)+\partial\vartheta(\bar x)
\end{equation}
(see \cite[Theorem 3.48]{MorNam22}).

\item[$\ $] chain rule: if $\Lambda\in\Lin(\X,\Y)$ and $\varphi:\Y\longrightarrow\R\cup\{\pm\infty\}$
is a convex function, continuous at $\Lambda(\bar x)$, with $\bar x\in\X$ , then it holds
\begin{equation}   \label{eq:subdchainr}
  \partial(\varphi\circ\Lambda)(\bar x)=\Lambda^*\left(\partial
  \varphi(\Lambda\bar x)\right)
\end{equation}
(see \cite[Theorem 3.55]{MorNam22}).
\end{itemize}

Let $\Omega\subseteq\X$ be a nonempty, closed, convex set and $\bar x
\in\X$.  In this case, the function $x\mapsto\dist{x}{\Omega}$ turns out to be
(Lipschitz) continuous and convex on $\X$, with $\dom\dist{\cdot}{\Omega}=\X$.
The generalized differentiation of this function leads to a connection between the
two above constructions:
\begin{equation}   \label{eq:subddist}
  \partial\dist{\cdot}{\Omega}(\bar x)=\left\{
  \begin{array}{ll}
    \Ncone{\bar x}{\Omega}\cap\Uball^*, & \quad\hbox{ if } \bar x\in\Omega, \\
    \\
    \Ncone{\bar x}{\cball{\Omega}{d_\Omega}}\cap\Usfer^*,
    & \quad\hbox{ if } \bar x\not\in\Omega,
  \end{array}\right.
\end{equation}
where $d_\Omega=\dist{\bar x}{\Omega}$
(see \cite[Proposition 3.77(a), Theorem 3.82]{MorNam22}).

A further element to be recalled enables to estimate the strong slope
of a convex function by means of its subdifferential: if
$\varphi:\X\longrightarrow\R\cup\{\pm\infty\}$ is a convex
function and $\bar x\in\dom\varphi$, then
\begin{equation}     \label{eq:stslsubdifest}
  \stsl{\varphi}(\bar x)=\inf\{\|x^*\|:\ x^*\in\partial\varphi(\bar x)\}
  =\dist{\nullv^*}{\partial\varphi(\bar x)}.
\end{equation}
(see \cite[Theorem 5(i)]{FaHeKrOu10}).

By combining the recalled subdifferential estimate with \cite[Theorem 2.8]{AzeCor06},
on account of the metric completeness of $\X$ (assumption $(a_0)$),
it is possible to establish the following global error bound result,
which will be employed in the sequel:

\begin{lemma}    \label{lem:sovebconvex}
Let $\varphi:\X\longrightarrow\R$ be a convex function continuous on
the Banach space $\X$.
Assume that $[\varphi>0]\ne\varnothing$ and that
\begin{equation}    \label{def:ebsolvconvex}
      \tau_\partial=\inf_{x\in [\varphi>0]}\dist{\nullv^*}{\partial\varphi(x)}
  =\inf\{\|x^*\|\ :\ x^*\in\partial\varphi(x),\quad x\in [\varphi>0]\}>0.
\end{equation}
Then, it is $[\varphi\le 0]\ne\varnothing$ and
$$
  \dist{x}{[\varphi\le 0]}\le {\varphi(x)\over\tau_\partial},
  \quad\forall x\in [\varphi>0].
$$
\end{lemma}

Recall that a set-valued map between normed spaces $F:\X\rightrightarrows\Y$
is said to be convex if $\graph F$ is a convex subset of $\X\times\Y$.

\begin{lemma}    \label{lem:convmfunct}
If $F:\X\rightrightarrows\Y$ is convex, then the function $(x,y)\mapsto\dist{y}{F(x)}$
is convex on $\X\times\Y$.
\end{lemma}

For a finite-dimensional setting, its proof can be found e.g. in \cite[Lemma 2.7]{Uder23}.
Nonetheless, the arguments employed there can be reproduced without any change in a normed
space setting.

\vskip.5cm


\section{Lipschitzian properties of $\Solv$ under metric perturbations}   \label{Sect:3}

\subsection{The dual regularity condition, local solvability and error bounds}

Given any $\delta>0$ and any pair $(p,x)\in P\times\X$, define by convenience $\mf_{A,Q}(p,x)=\dist{A(p,x)}{Q(p)}$,
$\mf_C(p,x)=\dist{x}{C(p)}$, $d_{A,Q}=\dist{A(p,x)}{Q(p)}$ and $d_C=\dist{x}{C(p)}$.
On the base of these notations, let us introduce the following quantities
\begin{eqnarray}   \label{eq:deftauAQ}
    \tau_{A,Q}(\delta)=\inf\bigg\{\|x^*\|:\ & x^*\in A(p,\cdot)^*\left(\Ncone{A(p,x)}{\cball{Q(p)}{d_{A,Q}})}\cap\Usfer^*\right) \nonumber \\
     & +[\Ncone{x}{\cball{C(p)}{d_C}}\cap\Uball^*],  \nonumber \\
    & (p,x)\in [\mf_{A,Q}>0]\cap\left[\ball{\bar p}{\delta}\times\ball{\bar x}{\delta}\right]\bigg\}
\end{eqnarray}
and
\begin{eqnarray}   \label{eq:deftauC}
   \tau_C(\delta)=\inf\bigg\{ & \|x^*\|:\  x^*\in A(p,\cdot)^*\left(\Ncone{A(p,x)}{Q(p)}\cap\Uball^*\right)
      +[\Ncone{x}{\cball{C(p)}{d_C}}\cap\Usfer^*],  \nonumber \\
    & (p,x)\in [\mf_{A,Q}=0]\cap[\mf_{C}>0]\cap\left[\ball{\bar p}{\delta}\times\ball{\bar x}{\delta}\right]\bigg\}.
\end{eqnarray}
Notice that in the formula (\ref{eq:deftauAQ}), whenever $x\in C(p)$, so $d_C=0$,
then it results in $\Ncone{x}{\cball{C(p)}{d_C}}=\Ncone{x}{C(p)}$.
In what follows, some conditions for the quantitative semicontinuity
properties of $\Solv$ are established under a qualification of the
behaviour of a family $(\SFP_p)$ near $(\bar p,\bar x)$, which is expressed
in terms of the constructions introduced above. More precisely,
a family $(\SFP_p)$ is said to satisfy the {\it dual regularity condition}
(for short, \DRC) near $(\bar p,\bar x) \in\graph\Solv$  if
$$
  \exists\delta>0:\ \tau(\delta)=\min\{\tau_{A,Q}(\delta),\,
  \tau_C(\delta)\}>0.   \leqno(\DRC)
$$

\vskip.5cm

\begin{theorem}     \label{thm:Solverbo}
With reference to a family $(\SFP_p)$, let $(\bar p,\bar x)\in
\graph\Solv$. Suppose that:
\begin{itemize}
  \item[(i)]  $C:P\rightrightarrows\X$ is l.s.c. at $(\bar p,\bar x)$;

  \item[(ii)] $Q:P\rightrightarrows\Y$ is l.s.c. at $\left(\bar p,A(\bar p,\bar x)\right)$;

  \item[(iii)] $A(\cdot,\bar x):P\rightarrow\Y$ is continuous at $\bar p$;

  \item[(iv)] the condition $(\DRC)$ holds.
\end{itemize}
Then, there exist $\eta,\, \zeta>0$ such that:
\begin{itemize}
  \item[(t)] $\Solv(p)\cap\cball{\bar x}{\eta}\ne\varnothing,
  \quad\forall p\in\cball{\bar p}{\zeta}$;

  \item[(tt)] the following estimate holds
\begin{equation}    \label{in:serbo}
  \dist{x}{\Solv(p)}\le\tau(\delta)^{-1}\mf(p,x),\quad\forall (p,x)\in
  \cball{\bar p}{\zeta}\times\cball{\bar x}{\eta}.
\end{equation}
\end{itemize}
\end{theorem}

\begin{proof}
In the light of the scalar characterization of $\Solv$ provided by
Proposition \ref{pro:scacharSolv}, both the assertions in the thesis
follow from Lemma \ref{lem:basicl} (the Basic Lemma).
Thus, it suffices to show that all the hypotheses of Lemma \ref{lem:basicl}
are satisfied, upon the current set of assumptions.
To this aim, observe that assumption $(a_0)$ ensures the metric
completeness of the space $(\X,\|\cdot\|)$.
The fact that $(\bar p,\bar x)\in\graph\Solv$ implies that $\mf(\bar p,\bar x)=0$.
By virtue of hypotheses (i)-(iii), Lemma \ref{lem:contmf} allows to state that
the function $p\mapsto\mf(p,\bar x)$ is continuous at $\bar p$.
Moreover, as observed in Remark \ref{rem:Lipcontmf}, each function
$x\mapsto\mf(p,x)$ is (Lipschitz) continuous, and hence, l.s.c. on $\X$,
for every $p\in\ball{\bar p}{r}$, with any radius $r>0$.
So, it remains to check the validity of the hypothesis (v) in Lemma \ref{lem:basicl}.
It is at this point that the condition $(\DRC)$ comes into play.
Take an arbitrary $(p,x)\in [\mf>0]\cap\left[\ball{\bar p}{\delta}\times
\ball{\bar x}{\delta}\right]$, where $\delta>0$ is as in the assumption (iv).

Consider first the case $\dist{A(p,x)}{Q(p)}>0$.
Notice that, for every $p\in P$, the function $x\mapsto\mf(p,x)$ is convex
and real-valued on $\X$. Thus, in calculating $\pastsl{\mf}(p,x)$,
which is a derivative-like object dealing with variations with
respect to the variable $x$ only, one can take profit of convexity.
By taking into account the subdifferential
estimate of the strong slope (\ref{eq:stslsubdifest}) and the subdifferential
sum rule (\ref{eq:subdsumr}), one obtains
\begin{eqnarray}   \label{eq:stsldist0subdifmf}
  \pastsl{\mf}(p,x) &=& \dist{\nullv^*}{\partial\mf(p,\cdot)(x)} \\
   &=& \dist{\nullv^*}{\partial\dist{A(p,\cdot)}{Q(p)}(x)+   \nonumber
   \partial\dist{\cdot}{C(p)}(x)}.
\end{eqnarray}
Since $A(p,x)\not\in Q(p)$, according to the second case in formula
(\ref{eq:subddist}) and the chain rule for subdifferential (\ref{eq:subdchainr}),
one can write
\begin{eqnarray*}
  \partial\dist{A(p,\cdot)}{Q(p)}(x) &=& A(p,\cdot)^*(\partial\dist{\cdot}{Q(p)}(A(p,x))) \\
   &=& A(p,\cdot)^*\left(\Ncone{A(p,x)}{\cball{Q(p)}{d_{A,Q}})}\cap\Usfer^*\right).
\end{eqnarray*}
On the other hand, if $x\in C(p)$, according to the first case in formula
(\ref{eq:subddist}), one has
$$
  \partial\dist{\cdot}{C(p)}(x)=\Ncone{x}{C(p)}\cap\Uball^*=
  \Ncone{x}{\cball{C(p)}{d_C}}\cap\Uball^*.
$$
Otherwise, if $x\not\in C(p)$, now $d_C>0$ so the second case in  formula
(\ref{eq:subddist}) gives
$$
  \partial\dist{\cdot}{C(p)}(x)=\Ncone{x}{\cball{C(p)}{d_C}}\cap\Usfer^*
  \subseteq\Ncone{x}{\cball{C(p)}{d_C}}\cap\Uball^*.
$$
Therefore, whatever is the case, the following inclusion holds
$$
  \partial\mf(p,\cdot)(x)\subseteq
  A(p,\cdot)^*\left(\Ncone{A(p,x)}{\cball{Q(p)}{d_{A,Q}})}\cap\Usfer^*\right)
  +\Ncone{x}{\cball{C(p)}{d_C}}\cap\Uball^*.
$$
From this inclusion, by recalling the estimate in (\ref{eq:stsldist0subdifmf})
and hypothesis (iv), it follows
\begin{eqnarray}
     \pastsl{\mf}& &  \hskip-0.2cm (p,x) \nonumber  \\
    & &\hskip-0.2cm\ge \dist{\nullv^*}{A(p,\cdot)^*\left(\Ncone{A(p,x)}{\cball{Q(p)}{d_{A,Q}}}\cap
    \Usfer^*\right)+\Ncone{x}{\cball{C(p)}{d_C}}\cap\Uball^*} \nonumber \\
    & & \hskip-0.2cm\ge \tau_{A,Q}(\delta)>0.   \label{in:pastslAQ}
\end{eqnarray}

Let us consider now the case $\dist{A(p,x)}{Q(p)}=0$. In such an event, it must be $\dist{x}{C(p)}>0$.
In other words, $(p,x)\in[\mf_{A,Q}=0]\cap[\mf_{C}>0]$.
By employing the same argument as in the previous case, with the difference that
this time one has $\partial\dist{A(p,\cdot)}{Q(p)}(x)=\Ncone{A(p,x)}{Q(p)}\cap\Uball^*$,
in the present case one obtains
\begin{eqnarray}
   \pastsl{\mf}(p,x)
    & \ge & \dist{\nullv^*}{A(p,\cdot)^*\left(\Ncone{A(p,x)}{Q(p)}\cap
    \Uball^*\right)+\Ncone{x}{\cball{C(p)}{d_C}}\cap\Usfer^*} \nonumber \\
   & &\hskip-0.2cm\ge \tau_{C}(\delta)>0.    \label{in:pastslC}
\end{eqnarray}
By arbitrariness of $(p,x)\in [\mf>0]\cap\left[\ball{\bar p}{\delta}\times
\ball{\bar x}{\delta}\right]$, the inequalities (\ref{in:pastslAQ}) and (\ref{in:pastslC})
enable one to conclude that $\tau_\nabla(\delta)\ge\tau(\delta)>0$,
thereby showing the validity of hypothesis (v)  in Lemma \ref{lem:basicl}.
Consequently, from the estimate in (\ref{in:merbo}) it is possible to
get immediately the estimate in (\ref{in:serbo}).
This completes the proof.
\smartqed\qed
\end{proof}

Before exploring those consequences of Theorem \ref{thm:Solverbo}, which are relevant
to the present analysis, some comment on its statement is due.
Under proper conditions, assertion (t) guarantees, as a qualitative result, solution existence for
every $(\SFP_p)$, provided that $p$ is close enough to $\bar p$, with
at least some solution of $(\SFP_p)$ lying not far from the nominal solution
$\bar x$ to $(\SFP_{\bar p})$. As a quantitative result, assertion (tt)
provides an error bound estimate for $(\SFP_p)$, which is valid all around
$(\bar p,\bar x)$, with the same constant $\tau(\delta)$.
A consequence of Theorem \ref{thm:Solverbo}, which is useful in the context
of approximated solution methods for $(\SFP)$, states that, under the described
assumptions, any solution to a given problem can be reached as a limit
of certain solution to perturbed problems.

\begin{corollary}
Let $(\bar p,\bar x)\in\graph\Solv$. Under the hypotheses of Theorem \ref{thm:Solverbo},
for every sequence $(p_n)_{n\in\N}$ in $P$, with $p_n\rightarrow\bar p$ as $n\to\infty$,
there exists a sequence $(x_n)_{n\in\N}$, with $x_n\in\Solv(p_n)$, converging to
$\bar x$.
\end{corollary}

\begin{proof}
It suffices to observe that the inequality in (\ref{in:serbo}), by assuming
without loss of generality that $p_n\in\cball{\bar p}{\zeta}$, allows to write
$$
  \dist{\bar x}{\Solv(p_n)}\le\tau(\delta)^{-1}\mf(p_n,\bar x),
  \quad\forall n\in\N.
$$
By taking into account that the function
$p\mapsto\mf(p,\bar x)$ is continuous at $\bar p$, as argued in the
proof of Theorem \ref{thm:Solverbo}, this inequality yields the existence of
a sequence $(x_n)_{n\in\N}$, such that $x_n\in\Solv(p_n)$ and
$$
  \|x_n-\bar x\|\le\tau(\delta)^{-1}\mf(p_n,\bar x)+\frac{1}{n}\
  \underset{n\to\infty}{\longrightarrow}\ \tau(\delta)^{-1}\mf(\bar p,\bar x)=0.
$$
\smartqed\qed
\end{proof}

\begin{remark}
Whenever $\X$ and $\Y$ are Hilbert spaces (in particular,
when they are finite-dimensional Euclidean spaces) the
formulation of $\tau_{A,Q}(\delta)$ and $\tau_{C}(\delta)$ can be
simplified. Indeed, in such a setting formula
(\ref{eq:subddist}) becomes
$$
  \partial\dist{\cdot}{\Omega}(\bar x)=\left\{
  \begin{array}{ll}
    \Ncone{\bar x}{\Omega}\cap\Uball  & \quad\hbox{ if } \bar x\in\Omega, \\
    \\
    \left\{\displaystyle{\bar x-\Proj{\bar x}{\Omega}\over \|\bar x-\Proj{\bar x}{\Omega}\|}\right\}
    & \quad\hbox{ if } \bar x\not\in\Omega,
  \end{array}\right.
$$
where $\Proj{\bar x}{\Omega}$ denotes the (uniquely defined) metric
projection of $\bar x$ onto $\Omega$ (see \cite[Corollary 3.79]{MorNam22}).
This fact yields
\begin{eqnarray*}
    \tau_{A,Q}(\delta)=\inf\bigg\{\|v\|:\ & v\in A(p,\cdot)^*\left({A(p,x)-\Proj{A(p,x)}{Q(p)}\over \|A(p,x)-\Proj{A(p,x)}{Q(p)}\|}\right)
      +[\Ncone{x}{\cball{C(p)}{d_C}}\cap\Uball],  \\
    & (p,x)\in [\mf_{A,Q}>0]\cap\left[\ball{\bar p}{\delta}\times\ball{\bar x}{\delta}\right]\bigg\}
\end{eqnarray*}
and
\begin{eqnarray*}
    \tau_{C}(\delta)=\inf\bigg\{\|v\|:\ & v\in A(p,\cdot)^*\left(\Ncone{A(p,x)}{Q(p)}\cap\Uball\right)
      +\left\{{x-\Proj{x}{C(p)}\over \|x-\Proj{x}{C(p)}\|}\right\},  \\
    & (p,x)\in  [\mf_{A,Q}=0]\cap [\mf_{C}>0]\cap\left[\ball{\bar p}{\delta}\times\ball{\bar x}{\delta}\right]\bigg\}.
\end{eqnarray*}
\end{remark}

In order to clarify the role of $(\DRC)$, some examples
are presented, which should discourage from its removal.
The first one aims at showing that the violation of the condition
$(\DRC)$ can yield dramatic effects already in terms of solution
existence.

\begin{example}     \label{ex:31}
Let $P=[0,+\infty)$ and $\X=\Y=\R$ be equipped with their standard metric
and normed structure, respectively. Consider the family $(\SFP_p)$
defined by the following data:
$$
   C(p)=(-\infty,-p],\qquad Q(p)=[p,+\infty),
   \qquad A(p,x)=\left(p+\frac{1}{2}\right)x,
$$
and set $\bar p=\bar x=0$. Clearly, the solution map $\Solv:[0,+\infty)
\rightrightarrows\R$ associated with the family defined above is given by
$$
  \Solv(p)=C(p)\cap A(\cdot,p)^{-1}(Q(p))=
  (-\infty,-p]\cap \left[\frac{p}{p+\frac{1}{2}},+\infty\right),
  \quad p\in [0,+\infty).
$$
Therefore, one finds
\begin{eqnarray*}
  \Solv(p)=\left\{\begin{array}{ll}
      \{0\}, & \qquad\hbox{\ if } p=0, \\
   \\
      \varnothing, & \qquad\hbox{\ if } p\in (0,+\infty). \\
  \end{array}
  \right.
\end{eqnarray*}
In other words, such an instance of $(\SFP_p)$ fails to be solvable
if $p\ne \bar p=0$, no matter how close is $p$ to $\bar p$.
Through the definition, one checks that $C$ and $Q$ are l.s.c. at $(0,0)$
and at $(0,A(0,0))=(0,0)$, respectively.
The map $A(\cdot,0)$ is continuous everywhere, as it is constant.
Let us check that the condition $(\DRC)$ in this case is violated.
Fix an arbitrary $\delta>0$ and choose $p_\delta\in (0,\min\{\delta,\frac{1}{2}\})$
and $x_\delta$ such that
$$
  -\delta<x_\delta=-p_\delta<0<\frac{p_\delta}{p_\delta+\frac{1}{2}}.
$$
With the above choices, it is true that
$$
  (p_\delta,x_\delta)\in\ball{0}{\delta}\times\ball{0}{\delta}
  \hbox{ and } A(p_\delta,x_\delta)\not\in [p_\delta,+\infty)=Q(p_\delta),
$$
and hence $(p_\delta,x_\delta)\in [\mf_{A,Q}>0]\cap[\ball{0}{\delta}\times\ball{0}{\delta}]$.
Nevertheless, one has
$$
  \Ncone{A(p_\delta,x_\delta)}{\cball{Q(p_\delta)}{d_{A,Q}}}\cap\Usfer=\{-1\},
$$
so that
$$
  A(p_\delta,\cdot)^*\left(\Ncone{A(p_\delta,x_\delta)}{\cball{Q(p_\delta)}{d_{A,Q}}}\cap\Usfer\right)
  =\left\{-\left(p_\delta+\frac{1}{2}\right)\right\}.
$$
On the other hand, since $x_\delta\in\bd (-\infty,-p_\delta]$, one has
$$
  \Ncone{x_\delta}{\cball{C(p_\delta)}{d_C}}\cap\Uball=
  \Ncone{x_\delta}{C(p_\delta)}\cap\Uball=[0,1].
$$
Thus, for an element such as $(x_\delta,p_\delta)$ one finds
\begin{eqnarray*}
  A(p_\delta,\cdot)^*\left(\Ncone{A(p_\delta,x_\delta)}{\cball{Q(p_\delta)}{d_{A,Q}}}\cap\Usfer\right)+
    [\Ncone{x_\delta}{\cball{C(p_\delta)}{d_C}}\cap\Uball]  \\
  =\left\{-\left(p_\delta+\frac{1}{2}\right)\right\}+[0,1]=
  \left[-\left(p_\delta+\frac{1}{2}\right),\frac{1}{2}-p_\delta\right].
\end{eqnarray*}
Since $p_\delta<\frac{1}{2}$, one has $0\in \left[-\left(p_\delta+\frac{1}{2}\right),\frac{1}{2}-p_\delta\right]$,
what implies that $\tau(\delta)=\tau_{A,Q}=0$.
As the same happens for every $\delta>0$, the condition (\DRC) is
violated.
\end{example}

The next example shows that, even if all the problems $(\SFP_p)$ are solvable all around
$\bar p$, yet in the absence of the condition $(\DRC)$ the error bound estimate
in (\ref{in:serbo}) may fail to hold true.

\begin{example}    \label{ex:32}
Let $P$, $\X$ and $\Y$ be as in Example \ref{ex:31}.
Consider now the family $(\SFP_p)$ defined by the following data:
$$
   C(p)=\R,\qquad Q(p)=[p^2,+\infty),
   \qquad A(p,x)=px,
$$
and set again $\bar p=\bar x=0$. The related solution map $\Solv$
is given by
\begin{eqnarray*}
  \Solv(p)=\left\{\begin{array}{ll}
      \R, & \qquad\hbox{\ if } p=0, \\
   \\
    \hbox{$[p,+\infty)$}, & \qquad\hbox{\ if } p\in (0,+\infty).
  \end{array}  \right.
\end{eqnarray*}
Consequently,
$$
  \dist{x}{\Solv(p)}=\max\{p-x,0\},\quad\forall (p,x)\in [0,+\infty)\times\R,
$$
while
$$
  \dist{A(p,x)}{Q(p)}=\max\{p^2-px,0\},\quad\forall (p,x)\in [0,+\infty)\times\R.
$$
In particular, once fixed arbitrarily positive $\eta,\, \zeta$ and $\sigma$, if taking
$p\in (0,\min\{\sigma^{-1},\, \zeta\})$, one finds
$$
\dist{0}{\Solv(p)}=p>\sigma p^2=\sigma\dist{A(p,0)}{Q(p)},
$$
which contradicts the assertion (tt) in Theorem \ref{thm:Solverbo}.
Notice that, whereas hypotheses (i)--(iii) of that theorem are fulfilled,
the condition $(\DRC)$ is violated. To see this, fix an arbitrary $\delta>0$
and consider the sequence of points $\{(1/n,x_n)\}_{n\in\N}$ in $P\times\X=[0,+\infty)\times\R$,
with $x_n$ such that
$$
  0<x_n<\frac{1}{n^2}<\frac{1}{n},\quad n\in\N\backslash\{0,\, 1\}.
$$
For $n>1/\delta$, one has $(1/n,x_n)\in\ball{0}{\delta}\times\ball{0}{\delta}$
and $A(1/n,x_n)\not\in Q(1/n)$, and hence $\Ncone{A(1/n,\cdot)}{\cball{ Q(1/n)}{d_{A,Q}}}\cap\Usfer=\{-1\}$,
so that
$$
  A(1/n,\cdot)^*\left(\Ncone{A(1/n,\cdot)}{\cball{ Q(1/n)}{d_{A,Q}}}\cap\Usfer\right)=\left\{-\frac{1}{n}\right\},
  \quad n\in\N\backslash\{0,\, 1\},
$$
while $\Ncone{x_n}{\cball{C(1/n)}{d_C}}=\Ncone{x_n}{C(1/n)}=\{0\}$, as $x_n\in\inte C(1/n)$. Since
for this family of problems $[\mf_C>0]=
\varnothing$, one has $\tau_C=+\infty$. Thus, one obtains
$$
   \tau(\delta)=\tau_{A,Q}(\delta)\le\lim_{n\to\infty}\left|-\frac{1}{n}\right|=0.
$$
\end{example}

In contrast to the preceding examples, the next one illustrates a simple case, to
which Theorem \ref{thm:Solverbo} can be successfully applied thanks to the fulfilment
of the condition $(\DRC)$.

\begin{example}
Let $P$, $\X$ and $\Y$ be as in Example \ref{ex:31}.
Consider the family $(\SFP_p)$ defined by the following data:
$$
   C(p)=\R,\qquad Q(p)=[p,+\infty),
   \qquad A(p,x)=(p+1)x,
$$
and set again $\bar p=\bar x=0$. Associated to such a family $(\SFP_p)$,
one finds the following solution map:
$$
   \Solv(p)=\left[\frac{p}{p+1},+\infty\right),\quad\forall p\in [0,+\infty).
$$
It follows
\begin{eqnarray}    \label{in:erboex:33}
  \dist{x}{\Solv(p)} &=&  \max\left\{\frac{p}{p+1}-x,\, 0\right\} \nonumber
   \le \max\{p-(p+1)x,\, 0\}   \\
   &=&\mf_{A,Q}(p,x)=\mf(p,x), \quad\forall (p,x)\in
   [0,+\infty)\times\R.
\end{eqnarray}
Indeed, if $x<\frac{p}{p+1}$, as $p\in [0,+\infty)$, it is true that
$$
  \frac{p}{p+1}-x\le p-(p+1)x.
$$
The above inequality shows that the error bound estimate in (\ref{in:serbo})
holds with $\tau(\delta)=1$ and arbitrary $\zeta$ and $\eta$.
Let us check the fulfilment of the condition $(\DRC)$.
If $(p,x)$ is an arbitrary element in $[0,+\infty)\times\R$, such that $\mf_{A,Q}(p,x)>0$, that is
$x<\frac{p}{p+1}$, one has $\Ncone{A(p,x)}{\cball{Q(p)}{d_{A,Q}}}\cap\Usfer=\{-1\}$, and hence
$$
  A(p,\cdot)^*\left(\Ncone{A(p,x)}{\cball{Q(p)}{d_{A,Q}}}\cap\Usfer\right)=\{-(p+1)\},
$$
while $\Ncone{x}{C(p)}\cap\Uball=\{0\}$, as $x\in\inte C(p)$. Consequently,
fixed any $\delta>0$, one obtains
$$
   \tau_{A,Q}(\delta)=\inf\{|-(p+1)|:\ (p,x)\in [\mf_{A,Q}>0]\cap
   [\ball{0}{\delta}\times\ball{0}{\delta}]\}=1.
$$
The fact that $[\mf_C>0]=\varnothing$ so $\tau_C(\delta)=+\infty$, allows
one to conclude that $\tau(\delta)=1$.
\end{example}

In order to afford further insights into the role played by $(\DRC)$, it is relevant to
note that the error bound estimate in (\ref{in:serbo}) can be achieved also
under an alternative qualification condition based on the concept of
subtransversality/metric subregularity.
Recall that two intersecting subsets, say $\Omega_1$ and $\Omega_2$, of the same
metric space are said to be subtransversal at $x_0\in\Omega_1\cap\Omega_2$
if there exist $\alpha\in (0,1)$ and $\delta>0$ such that
$$
  \alpha\dist{x}{\Omega_1\cap\Omega_2}\le\max\{\dist{x}{\Omega_1},\, \dist{x}{\Omega_2}\},
  \quad\forall x\in\cball{x_0}{\delta}.
$$
A set-valued map between metric spaces $F:X\rightrightarrows Y$ is said to be
subregular at $(x_0,y_0)\in\graph F$ if there exists positive $\kappa$ and $\delta$
such that
$$
  \dist{x}{F^{-1}(y_0)}\le\kappa\dist{y_0}{F(x)},\quad\forall
  x\in\cball{x_0}{\delta}.
$$
Now, with reference to a family $(\SFP_p)$, notice that if, for any fixed $p\in P$,
the set-valued map $F_{A,Q}(p,\cdot):\X\rightrightarrows\Y$, defined as
$$
   F_{A,Q}(p,x)=A(p,x)-Q(p),
$$
is metrically subregular at $(\bar x,\nullv)$ one has
\begin{eqnarray*}
  \dist{x}{A(p,\cdot)^{-1}(Q(p))} &=& \dist{x}{[A(p,\cdot)-Q(p)]^{-1}(\nullv)}=
  \dist{x}{F_{A,Q}(p,\cdot)^{-1}(\nullv)} \\
   &\le& \kappa \dist{\nullv}{ F_{A,Q}(p,x)}=\kappa \dist{\nullv}{A(p,x)-Q(p)}  \\
   &=& \kappa \dist{A(p,x)}{Q(p)},\quad\forall x\in\cball{\bar x}{\delta}.
\end{eqnarray*}
Thus, if assuming that there exist positive constants $\zeta$ and $\eta$
such that both the following condition hold:
\begin{itemize}
  \item[(i)] for every $p\in\cball{\bar p}{\zeta}$, each pair of sets $C(p)$ and
  $A(p,\cdot)^{-1}(Q(p))$ do intersect in $\cball{\bar x}{\eta}$ and
  are subtransversal at $\bar x\in C(\bar p)\cap A(\bar p,\cdot)^{-1}(Q(\bar p))$,
  with uniform subtransversality constant $\alpha\in (0,1)$ and $\delta>0$;

  \item[(ii)] for every $p\in\cball{\bar p}{\zeta}$, each set-valued map
  $F_{A,Q}(p,\cdot)$ is metrically subregular at $(\bar x,\nullv)$, with a uniform
  metric subregularity bound $\kappa>0$,
\end{itemize}
one finds
\begin{eqnarray*}
  \dist{x}{\Solv(p)} &=& \dist{x}{C(p)\cap[A(p,\cdot)^{-1}(Q(p))]}   \\
  &\le & \alpha\max\{\dist{x}{C(p)},\, \dist{x}{A(p,\cdot)^{-1}(Q(p))}\} \\
   &\le& \alpha \left[\dist{x}{C(p)}+\dist{x}{A(p,\cdot)^{-1}(Q(p))}\right] \\
  &\le&  \alpha \left[\dist{x}{C(p)}+\kappa \dist{A(p,x)}{Q(p)}\right] \\
  &\le & \alpha\max\{1,\, \kappa\}\mf(p,x), \quad\forall
  (p,x)\in  \cball{\bar p}{\zeta}\times\cball{\bar x}{\eta}.
\end{eqnarray*}
Since a set-valued map $F:\X\rightrightarrows\Y$ between normed
spaces is known to be metrically subregular at $(x_0,y_0)\in\graph F$ iff
the pair $\graph F$ and $\X\times\{y_0\}$ is subtransversal at $(x_0,y_0)$,
one sees that the alternative qualification condition presented above can be expressed
purely in terms of subtransversality (for a thorough study of subtransversality
in the context of variational analysis, the reader can refer to \cite[Section 7.1]{Ioff17}
and \cite{KrLuTh17}).


\subsection{A condition for Lipschitz lower semicontinuity}

On the base of the elements discussed above, one is in a
position to establish the following sufficient condition for the first of the
properties appearing in Definition \ref{def:Lipsemcont}.

\begin{theorem}[Lipschitz lower semicontinuity of $\Solv$]   \label{thm:SolvLiplsc}
With reference to a family $(\SFP_p)$, let $(\bar p,\bar x)\in
\graph\Solv$. Suppose that:
\begin{itemize}
  \item[(i)]  $C:P\rightrightarrows\X$ is Lipschitz l.s.c. at $(\bar p,\bar x)$;

  \item[(ii)] $Q:P\rightrightarrows\Y$ is Lipschitz l.s.c. at
  $\left(\bar p,A(\bar p,\bar x)\right)$;

  \item[(iii)] $A(\cdot,\bar x):P\rightarrow\Y$ is calm at $\bar p$;

  \item[(iv)] the condition $(\DRC)$ holds.
\end{itemize}
Then,  $\Solv$ is Lipschitz l.s.c. at $(\bar p,\bar x)$ and it holds
\begin{equation}    \label{in:LiplscSolvest}
    \Liplsc{\Solv}{\bar p}{\bar x}\le {\Liplsc{C}{\bar p}{\bar x}
    +\clm{A(\cdot,\bar x)}{\bar p}+\Liplsc{Q}{\bar p}{A(\bar p,\bar x)}
    \over \tau(\delta)}.
\end{equation}
\end{theorem}

\begin{proof}
Let us start with observing that, since  Lipschitz lower
semicontinuity implies mere lower semicontinuity for any set-valued
mapping (recall Remark \ref{rem:lscLiplsc}(i)), as well as calmness
for a single-valued mapping implies continuity (recall Remark \ref{rem:lscLiplsc}(iv)),
the assumption set of Theorem \ref{thm:SolvLiplsc} enables one to apply Theorem
\ref{thm:Solverbo}, according to which there exist positive
constants $\zeta$ and $\eta$ such that the estimate in (\ref{in:serbo})
is true. The latter, in particular, implies
\begin{equation}    \label{in:erboatbarx}
  \dist{\bar x}{\Solv(p)}\le\tau(\delta)^{-1}\mf(p,\bar x),\quad\forall p\in
  \cball{\bar p}{\zeta}.
\end{equation}
Take $\ell_C>\tilde{\ell_C}>\Liplsc{C}{\bar p}{\bar x}$. According to hypothesis (i)
there exists $r_C>0$ such that
\begin{equation}   \label{int:LiplscC}
  C(p)\cap \cball{\bar x}{\tilde{\ell_C} d(p,\bar p)}\ne\varnothing,\quad\forall p\in
  \cball{\bar p}{r_C}.
\end{equation}
Take $\ell_Q>\tilde{\ell_Q}>\Liplsc{Q}{\bar p}{A(\bar p,\bar x)}$. Owing to hypothesis (ii)
there exists $r_Q>0$ such that
\begin{equation}   \label{int:LiplscQ}
  Q(p)\cap \cball{A(\bar p,\bar x)}{\tilde{\ell_Q}d(p,\bar p)}\ne\varnothing,\quad\forall p\in
  \cball{\bar p}{r_Q}.
\end{equation}
Take $\gamma_A>\clm{A(\cdot,\bar x)}{\bar p}$. According to hypothesis (iii)
there exists $r_A>0$ such that
\begin{equation}   \label{in:clmAbarx}
  \|A(p,\bar x)-A(\bar p,\bar x)\|\le\gamma_Ad(p,\bar p),\quad\forall p\in
  \cball{\bar p}{r_A}.
\end{equation}
Now, define $r_\Sigma=\min\{\zeta,\, r_C,\, r_Q,\, r_A\}$.
On account of the nonemptiness in (\ref{int:LiplscC}), one has
\begin{equation}   \label{in:distC}
  \dist{\bar x}{C(p)}\le \tilde{\ell_C} d(p,\bar p),\quad\forall p
  \in\cball{\bar p}{r_\Sigma}.
\end{equation}
On account of the nonemptiness in (\ref{int:LiplscQ}) and of inequality
(\ref{in:clmAbarx}), one has
\begin{eqnarray}   \label{in:distAC}
  \dist{A(p,\bar x)}{Q(p)}&\le&\|A(p,\bar x)-A(\bar p,\bar x)\|+\dist{A(\bar p,\bar x)}{Q(p)}  \nonumber  \\
  &\le& (\gamma_A+\tilde{\ell_Q})d(p,\bar p), \quad\forall p
  \in\cball{\bar p}{r_\Sigma}.
\end{eqnarray}
Thus, by combining the inequalities (\ref{in:erboatbarx}), (\ref{in:distC})
and (\ref{in:distAC}), one obtains
\begin{eqnarray*}
  \dist{\bar x}{\Solv(p)} &\le &
  {\tilde{\ell_C}+\gamma_A+\tilde{\ell_Q}\over \tau(\delta)}d(p,\bar p)   \\
  &<& {\ell_C+\gamma_A+\ell_Q\over \tau(\delta)}d(p,\bar p),\quad\forall p
  \in\cball{\bar p}{r_\Sigma}\backslash\{\bar p\}.
\end{eqnarray*}
The last inequality means that
$$
  \Sigma(p)\cap \cball{\bar x}{\frac{\ell_C+\gamma_A+\ell_Q}{\tau(\delta)}d(p,\bar p)}
  \ne\varnothing,\quad\forall p\in
  \cball{\bar p}{r_\Sigma}.
$$
This amounts to say that $\Solv$ is Lipschitz l.s.c. at $(\bar p,\bar x)$
with
$$
  \Liplsc{\Solv}{\bar p}{\bar x}\le \frac{\ell_C+\gamma_A+\ell_Q}{\tau(\delta)}.
$$
Since the above inequality holds true for every $\ell_C>\Liplsc{C}{\bar p}{\bar x}$, for
every $\gamma_A>\clm{A(\cdot,\bar x)}{\bar p}$ and for every $\ell_Q>\Liplsc{Q}{\bar p}{A(\bar p,\bar x)}$,
from it one can derive the estimate (\ref{in:LiplscSolvest}) in the thesis.
This completes the proof.
\smartqed\qed
\end{proof}

\begin{remark}
With respect to Theorem \ref{thm:Solverbo}, the gain obtained thanks to Theorem \ref{thm:SolvLiplsc}
consists in ensuring the existence of a sequence $(x_n)_{n\in\N}$ of solutions to approximated problems
which converges to the solution $\bar x$ linearly with respect to the perturbation
parameter, with a linear convergence rate estimated as
in (\ref{in:LiplscSolvest}). Of course, the price to be paid for such an improvement
is a stricter set of assumptions, replacing mere semicontinuity properties of
the data with their quantitative counterparts.
\end{remark}


\subsection{A condition for calmness}

As Lipschitz lower semicontinuity and calmness are, in general,
independent of each other, a condition devoted to the latter property
is worth being established too.

\begin{theorem}[Calmness of $\Solv$]    \label{thm:Solvcalm}
With reference to a class $(\SFP_p)$, let $(\bar p,\bar x)\in\graph\Solv$.
Suppose that:
\begin{itemize}
  \item[(i)]  $C:P\rightrightarrows\X$ is l.s.c. and calm at $(\bar p,\bar x)$;

  \item[(ii)] $Q:P\rightrightarrows\Y$ is l.s.c. and calm at
  $\left(\bar p,A(\bar p,\bar x)\right)$;

  \item[(iii)] $A:P\times\X\rightarrow\Y$ is Lipschitz around $(\bar p,\bar x)$;

  \item[(iv)] the condition $(\DRC)$ holds.
\end{itemize}
Then,  $\Solv$ is calm at $(\bar p,\bar x)$ and it holds
\begin{equation}    \label{in:calmSolvest}
    \clmsv{\Solv}{\bar p}{\bar x}\le {\clmsv{C}{\bar p}{\bar x}
    +\lip{A}{(\bar p,\bar x)}+\clmsv{Q}{\bar p}{A(\bar p,\bar x)}
    \over \tau(\delta)}.
\end{equation}
\end{theorem}

\begin{proof}
Under the above set of assumptions it is possible to apply Theorem \ref{thm:Solverbo}.
Thus, one gets the existence of $\zeta,\, \eta>0$ for which the estimate in (\ref{in:serbo})
holds.
Take $\gamma_C>\clmsv{C}{\bar p}{\bar x}$.
According to hypothesis (i), there exist positive $r_C$ and $\eta_C$ such that
\begin{equation}   \label{in:clmC}
  C(p)\cap\cball{\bar x}{\eta_C}\subseteq\cball{C(\bar p)}{\gamma_C d(p,\bar p)},
  \quad\forall p\in\cball{\bar p}{r_C}.
\end{equation}
Take $\gamma_Q>\clmsv{Q}{\bar p}{A(\bar p,\bar x)}$.
According to hypothesis (ii), there exist positive $r_Q$ and $\eta_Q$ such that
\begin{equation}   \label{in:clmQ}
  Q(p)\cap\cball{A(\bar p,\bar x)}{\eta_Q}\subseteq\cball{Q(\bar p)}{\gamma_Q d(p,\bar p)},
  \quad\forall p\in\cball{\bar p}{r_Q}. \end{equation}
Take $\ell_A>\lip{A}{(\bar p,\bar x)}$. Then, by hypothesis (iii) there exists
$r_A>0$ such that
\begin{equation}    \label{in:lipA}
  \|A(p,x)-A(\bar p,\bar x)\|\le\ell_A\left[d(p,\bar p)+\|x-\bar x\|\right],\quad\forall
  (p,x)\in\cball{\bar p}{r_A}\times\cball{\bar x}{r_A}.
\end{equation}
Define
$$
 \eta_\Solv=\min\{\eta,\, \eta_C,\, \eta_Q\}\qquad\hbox{and}\qquad
 r_\Solv=\min\{\zeta,\, r_C,\, r_Q,\, r_A,\, (2\ell_A)^{-1}\eta_Q\}.
$$
Now, take an arbitrary $x\in\Solv(p)\cap\cball{\bar x}{\eta_\Solv}$.
Since, in particular, $x\in C(p)\cap\cball{\bar x}{\eta_C}$, then, on account of
inclusion (\ref{in:clmC}), one has
$$
   \dist{x}{C(\bar p)}\le\gamma_C d(p,\bar p),\quad\forall
   p\in\cball{\bar p}{r_\Solv}.
$$
Moreover, as $x\in\Solv(p)$, it must be $A(p,x)\in Q(p)$.
From inequality (\ref{in:lipA}), provided that $(p,x)\in \cball{\bar p}{r_\Solv}
\times\cball{\bar x}{r_\Solv}$, it is possible to deduce
$$
   \|A(p,x)-A(\bar p,\bar x)\|\le\ell_A\cdot 2r_\Solv\le\eta_Q,
$$
so it is true that $A(p,x)\in Q(p)\cap\cball{A(\bar p,\bar x)}{\eta_Q}$.
Consequently, by inequality (\ref{in:clmQ}) one finds
$$
  \dist{A(p,x)}{Q(\bar p)}\le\gamma_Q d(p,\bar p),
  \quad\forall p\in\cball{\bar p}{r_\Solv}.
$$
Thus, for every $x\in\Solv(p)\cap\cball{\bar x}{\eta_\Solv}$,
by recalling the estimate in (\ref{in:serbo}) one obtains
\begin{eqnarray*}
  \dist{x}{\Solv(\bar p)} &\le& \tau(\delta)^{-1}\mf(x,\bar p)=\tau(\delta)^{-1}
  \left[\dist{A(\bar p,x)}{Q(\bar p)}+\dist{x}{C(\bar p)}\right]  \\
   &\le & \tau(\delta)^{-1}
  \left[\|A(\bar p,x)-A(p,x)\|+\dist{A(p,x)}{Q(\bar p)}+\dist{x}{C(\bar p)}\right]  \\
  &\le & \tau(\delta)^{-1}\left(\ell_A+\gamma_Q+\gamma_C\right)d(p,\bar p),
  \quad\forall p\in\cball{\bar p}{r_\Solv}.
\end{eqnarray*}
By arbitrariness of $x\in\Solv(p)\cap\cball{\bar x}{\eta_\Solv}$, the last
inequality implies that
$$
 \Solv(p)\cap\cball{\bar x}{\eta_\Solv}\subseteq
 \cball{\Solv(\bar p)}{{\ell_A+\gamma_Q+\gamma_C\over \tau(\delta)}},
  \quad\forall p\in\cball{\bar p}{r_\Solv}.
$$
This shows that $\Solv$ is calm at $(\bar p,\bar x)$, with
$$
  \clmsv{\Solv}{\bar p}{\bar x}\le {\gamma_Q+\ell_A+\gamma_C
    \over \tau(\delta)}.
$$
As the above inequality remains valid for every $\gamma_Q>\clmsv{Q}{\bar p}{A(\bar p,\bar x)}$,
every $\ell_A>\lip{A}{(\bar p,\bar x)}$, and every $\gamma_C>\clmsv{C}{\bar p}{\bar x}$,
one can conclude that the estimate in (\ref{in:calmSolvest}) is true, thereby
completing the proof.
\smartqed\qed
\end{proof}

To complement Theorem \ref{thm:Solvcalm} by considering a special manifestation
of calmness, recall that, whenever a set-valued
map, which is calm at a given point in its graph, admits a graphical localization
that is single-valued at the same point, the map is said to satisfy the {\it isolated
calmness property}. Such a strengthened version of calmness attracted  a specific
interest in variational analysis inasmuch as, in contrast to mere calmness, it displays
a certain stability behaviour under perturbation (see \cite[Section 3.9]{DonRoc14}).
From Theorem  \ref{thm:Solvcalm} it is possible to derive a sufficient condition
for the isolated calmness of $\Solv$ as follows.

\begin{corollary}[Isolated calmness of $\Solv$]    \label{cor:isocalmSolv}
Under the hypotheses of Theorem  \ref{thm:Solvcalm}, suppose, in addition, that either

\begin{itemize}
  \item[$(\textrm{v}_1)$] $C:P\rightrightarrows\X$ has the isolated calmness property
  at $(\bar p,\bar x)$;

or

  \item[$(\textrm{v}_2)$] $Q:P\rightrightarrows\Y$ has the isolated calmness property
  at $(\bar p,A(\bar p,\bar x))$ and $A(\bar p,\cdot)$ is injective.
\end{itemize}
Then, $\Solv$ inherits the isolated calmness property at $(\bar p,\bar x)$.
\end{corollary}

\begin{proof}
By Theorem  \ref{thm:Solvcalm} $\Solv$ is calm at $(\bar p,\bar x)$.
If $(\textrm{v}_1)$ holds, then there exists $\eta_0>0$ such that
$$
  C(\bar p)\cap\cball{\bar x}{\eta_0}=\{\bar x\}.
$$
Therefore, one finds
$$
  \Solv(\bar p)\cap\cball{\bar x}{\eta_0}= [C(\bar p)\cap A(\bar p,\cdot)^{-1}
  (Q(\bar p))]\cap\cball{\bar x}{\eta_0} \subseteq
   C(\bar p)\cap\cball{\bar x}{\eta_0}
   =\{\bar x\}.
$$
Otherwise, if ($\textrm{v}_2$) holds, there exists $\eta_Q>0$ such that
\begin{equation}  \label{eq:isocalmQ}
   Q(\bar p)\cap\cball{A(\bar p,\bar x)}{\eta_Q}=\{A(\bar p,\bar x)\}.
\end{equation}
Then, take $\eta_0$ such that $0<\eta_0<\|A(\bar p,\cdot)\|_\Lin^{-1}\eta_Q$.
Observe that
$$
   \cball{\bar x}{\eta_0}\cap A(\bar p,\cdot)^{-1}(Q(\bar p))\subseteq
   A(\bar p,\cdot)^{-1}\left(Q(\bar p)\cap\cball{A(\bar p,\bar x)}{\eta_Q}\right).
$$
Indeed, if $v\in\cball{\bar x}{\eta_0}\cap A(\bar p,\cdot)^{-1}(Q(\bar p))$, one has
$A(\bar p,\bar x)\in Q(\bar p)$ and
$$
  \|A(\bar p,v)-A(\bar p,\bar x)\|\le\|A(\bar p,\cdot)\|_\Lin\|v-\bar x\|\le
  \|A(\bar p,\cdot)\|_\Lin\frac{\eta_Q}{\|A(\bar p,\cdot)\|_\Lin}=\eta_Q.
$$
Therefore, owing to the equality in (\ref{eq:isocalmQ}) and the injectivity
of $A(\bar p,\cdot)$, one finds
\begin{eqnarray*}
  \Solv(\bar p)\cap\cball{\bar x}{\eta_0}  &=& [C(\bar p)\cap A(\bar p,\cdot)^{-1}
  (Q(\bar p))]\cap\cball{\bar x}{\eta_0}
  \subseteq A(\bar p,\cdot)^{-1}\left(Q(\bar p)\right)\cap\cball{\bar x}{\eta_0}  \\
   &\subseteq & A(\bar p,\cdot)^{-1}\left(Q(\bar p)\cap\cball{A(\bar p,\bar x)}{\eta_Q}\right) \\
   &\subseteq & A(\bar p,\cdot)^{-1}(A(\bar p,\bar x))=\{\bar x\}.
\end{eqnarray*}
This means that $\Solv$ has the isolated calmness property at $(\bar p,\bar x)$, thereby
completing the proof.
\smartqed\qed
\end{proof}


\subsection{A condition for the Aubin property}

In consideration of its importance, a specific result for the Aubin
property of the solution map is next presented.

\begin{theorem}[Aubin property of $\Solv$]    \label{thm:SolvAubin}
With reference to a class $(\SFP_p)$, let $(\bar p,\bar x)\in\graph\Solv$.
Suppose that:
\begin{itemize}
  \item[(i)]  $C:P\rightrightarrows\X$ has the Aubin property around $(\bar p,\bar x)$;

  \item[(ii)] $Q:P\rightrightarrows\Y$ has the Aubin property around $(\bar p,A(\bar p,\bar x))$;

  \item[(iii)] $A:P\times\X\rightarrow\Y$ is Lipschitz around $(\bar p,\bar x)$;

  \item[(iv)] the condition $(\DRC)$ holds.
\end{itemize}
Then,  $\Solv$ has the Aubin property at $(\bar p,\bar x)$ and it holds
\begin{equation}    \label{in:AubinSolvest}
    \Lip{\Solv}{\bar p}{\bar x}\le {\Lip{C}{\bar p}{\bar x}
    +\lip{A}{(\bar p,\bar x)}+\Lip{Q}{\bar p}{A(\bar p,\bar x)}\over \tau(\delta)}.
\end{equation}
\end{theorem}

\begin{proof}
Fixed any $\ell_C>\Lip{C}{\bar p}{\bar x}$, by hypothesis (i) there exist
positive $\eta_C$ and $\delta_C$ such that
\begin{equation}    \label{in:AubproC}
  C(p_1)\cap\cball{\bar x}{\eta_C}\subseteq\cball{C(p_2)}{\ell_Cd(p_1,p_2)},
  \quad\forall p_1,\, p_2\in\cball{\bar p}{\delta_C}.
\end{equation}
Similarly, fixed any $\ell_Q>\Lip{Q}{\bar p}{A(\bar p,\bar x)}$, by hypothesis (ii) there exist
positive $\eta_Q$ and $\delta_Q$ such that
\begin{equation}    \label{in:AubproQ}
  Q(p_1)\cap\cball{A(\bar p,\bar x)}{\eta_Q}\subseteq\cball{Q(p_2)}{\ell_Qd(p_1,p_2)},
  \quad\forall p_1,\, p_2\in\cball{\bar p}{\delta_Q}.
\end{equation}
Fixed any $\ell_A>\lip{A}{(\bar p,\bar x)}$, by hypothesis (iii) there exists $\delta_A>0$
such that
\begin{eqnarray}    \label{in:lipproA}
  \|A(p_1,x_1)-A(p_2,x_2)\| &\le& \ell_A[d(p_1,p_2)+\|x_1-x_2\|],    \\
  & & \forall p_1,\, p_2\in\cball{\bar p}{\delta_A},\
  \forall x_1,\, x_2\in\cball{\bar x}{\delta_A}.   \nonumber
\end{eqnarray}
Notice that, since the Aubin property at a given pair of the graph implies
the Lipschitz lower semicontinuity, and hence the lower semicontinuity at
the same pair (remember Remark \ref{rem:lscLiplsc}(i)), and the Lipschitz
continuity of $A$ around $(\bar p,\bar x)$ implies the continuity of
$A(\bar p,\cdot)$ at $\bar p$, then under the condition $(\DRC)$ one is in a
position to apply Theorem \ref{thm:Solverbo}, which guarantees the existence
of positive $\eta$ and $\zeta$ satisfying (t) and (tt).
In particular, from the estimate in (tt) it follows
\begin{eqnarray}     \label{in:Aubproerboxp2}
  \dist{x}{\Solv(p_2)} &\le& \tau(\delta)^{-1}[\dist{A(p_2,x)}{Q(p_2)}+
  \dist{x}{C(p_2)}] \\
   & & \forall p_2\in\cball{\bar p}{\zeta},\ \forall x\in\cball{\bar x}{\eta}.  \nonumber
\end{eqnarray}
Now, define
\begin{equation}   \label{eq:defetazeta}
  \eta_\Solv=\min\left\{\eta_C,\eta,\frac{\eta_Q}{2\ell_A},\delta_A\right\} \quad\hbox{ and }\quad
  \delta_\Solv=\min\{\delta_C,\delta_Q,\delta_A,\zeta,\eta_\Solv\}.
\end{equation}
Take an arbitrary $x\in\Solv(p_1)\cap\cball{\bar x}{\eta_\Solv}$, with $p_1\in\cball{\bar p}{\delta_\Solv}$.
By definition of $\Solv(p_1)$, it must be $x\in C(p_1)$ and hence, according
to the choice of $\eta_\Solv$, one has $x\in C(p_1)\cap\cball{\bar x}{\eta_C}$.
From the inclusion in (\ref{in:AubproC}), as $\delta_\Solv\le\delta_C$, it follows
\begin{equation}    \label{in:distxCp2}
  \dist{x}{C(p_2)}\le\ell_C d(p_1,p_2),\quad\forall p_2\in\cball{\bar p}{\delta_\Solv}.
\end{equation}
Besides, again by definition of $\Solv(p_1)$, it must be $A(p_1,x)\in Q(p_1)$.
By virtue of the choice of $\delta_\Solv$ and $\eta_\Solv$, as $x\in
\cball{\bar x}{\eta_\Solv}\subseteq\cball{\bar x}{\delta_A}$,
from the inequality in (\ref{in:lipproA}) it follows
\begin{equation}   \label{in:Ap1Ap2x}
  \|A(p_1,x)-A(p_2,x)\|\le\ell_A d(p_1,p_2),\quad\forall p_1,\, p_2\in
  \cball{\bar p}{\delta_\Solv}.
\end{equation}
Moreover, again from (\ref{in:lipproA}) and the definitions of $\eta_\Solv$
and $\delta_\Solv$ in (\ref{eq:defetazeta}), one deduces
$$
  \|A(p_1,x)-A(\bar p,\bar x)\| \le  \ell_A [d(p_1,\bar p)+\|x-\bar x\|]
  \le \ell_A(\delta_\Solv+\eta_\Solv) \le 2\ell_A\eta_\Solv\le\eta_Q,
$$
wherefrom
$$
  A(p_1,x)\in Q(p_1)\cap\cball{A(\bar p,\bar x)}{\eta_Q},\quad\forall
  p_1\in\cball{\bar p}{\delta_\Solv}.
$$
On account of (\ref{in:AubproQ}), as $\delta_\Solv\le\delta_Q$, the last
inclusion enables one to write
\begin{equation}     \label{in:distAp1Qp2}
    \dist{A(p_1,x)}{Q(p_2)}\le\ell_Q d(p_1,p_2),\quad\forall
    p_1,\, p_2\in\cball{\bar p}{\delta_\Solv}.
\end{equation}
By combining the inequalities in (\ref{in:Ap1Ap2x}) and in (\ref{in:distAp1Qp2}),
one obtains
\begin{eqnarray*}
  \dist{A(p_2,x)}{Q(p_2)} &\le& \|A(p_1,x)-A(p_2,x)\|+\dist{A(p_1,x)}{Q(p_2)} \nonumber \\
   &\le& (\ell_A+\ell_Q)d(p_1,p_2),\quad\forall p_1,\, p_2\in\cball{\bar p}{\delta_\Solv}.
\end{eqnarray*}
In the light of (\ref{in:Aubproerboxp2}) and (\ref{in:distxCp2}), the last inequality gives
$$
 \dist{x}{\Solv(p_2)} \le \tau(\delta)^{-1}[\ell_A+\ell_Q+\ell_C]
 d(p_1,p_2),\quad\forall p_1,\, p_2\in\cball{\bar p}{\delta_\Solv}.
$$
By arbitrariness of $x\in\Solv(p_1)\cap\cball{\bar x}{\eta_\Solv}$, the
above inequality shows that
$$
  \Solv(p_1)\cap\cball{\bar x}{\eta_\Solv}\subseteq\cball{\Solv(p_2)}{{\ell_A+\ell_Q+\ell_C\over
  \tau(\delta)}d(p_1,p_2)},\quad\forall p_1,\, p_2\in\cball{\bar p}{\delta_\Solv},
$$
which proves the first assertion in the thesis, with
$$
  \Lip{\Solv}{\bar p}{\bar x}\le\frac{\ell_A+\ell_Q+\ell_C}{\tau(\delta)}.
$$
The estimate in (\ref{in:AubinSolvest}) follows at once from the arbitrariness of
$\ell_C>\Lip{C}{\bar p}{\bar x}$, $\ell_Q>\Lip{Q}{\bar p}{A(\bar p,\bar x)}$ and
$\ell_A>\lip{A}{(\bar p,\bar x)}$. This completes the proof.
\smartqed\qed
\end{proof}

\begin{corollary}
Under the hypotheses of Theorem \ref{thm:SolvAubin}, suppose in addition
that either assumption $(\textrm{v}_1)$ or $(\textrm{v}_2)$ in Corollary
\ref{cor:isocalmSolv} holds. Then, there exists $\eta>0$ such that
for every sequence $(p_n)_{n\in\N}$ in $P$, with $p_n\rightarrow\bar p$ as $n\to\infty$,
every sequence $(x_n)_{n\in\N}$ in $\cball{\bar x}{\eta}$,
with $x_n\in\Solv(p_n)$, converges to $\bar x$ linearly with respect
to $d(p_n,\bar p)$, with rate estimated as in (\ref{in:AubinSolvest}).
\end{corollary}

\begin{proof}
The thesis follows from the fact that, in the present circumstance,
$\Solv$ has both the Aubin property and the isolated calmness property.
Thus, if $\eta>0$ is as in the Definition \ref{def:Lipsemcont}(iv)
and $\ell>\Lip{\Solv}{\bar p}{\bar x}$, since $\Solv(\bar p)=\{\bar x\}$,
for every sequence $(p_n)_{n\in\N}$
converging to $\bar p$, one has
$$
  \Solv(p_n)\cap\cball{\bar x}{\eta}\subseteq
  \cball{\bar x}{\ell d(p_n,\bar p)},
$$
which gives
$$
  \|x_n-\bar x\|\le\ell d(p_n,\bar p),
$$
for every sequence $(x_n)_{n\in\N}$ in $\Solv(p_n)\cap\cball{\bar x}{\eta}$.
\smartqed\qed
\end{proof}

The next example shows that the condition $(\DRC)$ is too strong to
yield a full characterization of the Aubin property of $\Solv$.

\begin{example}
Let $P=\X=\Y=\R$ be equipped with its standard (Euclidean) structure.
Consider the family $(\SFP_p)$ defined by the following data
$$
   C(p)=[-|p|,+\infty),\qquad Q(p)=[|p|p^2,+\infty),\qquad A(p,x)=p^2x,
$$
and fix $\bar p=\bar x=0$. This family admits as a solution map
the multifunction $\Solv:\R\rightrightarrows\R$ given by
$$
   \Solv(p)=[|p|,+\infty),\quad\forall p\in\R.
$$
Since $C$, $Q$ and $\Solv$ are the epigraphical map associated to
the functions $p\mapsto -|p|$, $p\mapsto |p|p^2$, and $p\mapsto |p|$,
respectively, which are all locally Lipschitz around $\bar p=0$,
they are locally Lipschitz continuous around $0$ as a set-valued map,
and hence satisfy the Aubin property around $(0,0)$.
Since $A\in C^1(\R^2)$, it is Lipschitz continuous around $(0,0)$.
Let us see what happens with the condition $(\DRC)$.
Fix an arbitrary $\delta>0$. Choose $(p,x)\in\ball{0}{\delta}\times
\ball{0}{\delta}$ such that $p\in (0,\delta)$ and $x\in (-p,p)$.
Then, since $x<|p|$, one has $A(p,x)=p^2x<p^2|p|$, so
$A(p,x)\not\in Q(p)$ and $(p,x)\in [\mf_{A,Q}>0]\cap [\ball{0}{\delta}\times
\ball{0}{\delta}]$. Therefore, one obtains
$$
  \Ncone{A(p,x)}{\cball{Q(p)}{d_{A,Q}}}\cap\Usfer=\{-1\}
$$
and consequently
$$
   A(p,\cdot)^*\left(\Ncone{A(p,x)}{\cball{Q(p)}{d_{A,Q}}}\cap\Usfer\right)
   =\{-p^2\}.
$$
On the other hand, as $x>-p=-|p|$, one has $x\in\inte C(p)$, giving
$$
  \Ncone{x}{\cball{C(p)}{d_C}}\cap\Uball=\Ncone{x}{C(p)}\cap\Uball=\{0\}.
$$
According to (\ref{eq:deftauAQ}), it follows that $\tau_{A,Q}(\delta)=0$, which
implies $\tau(\delta)\le\tau_{A,Q}(\delta)=0$.
As this happens for every $\delta>0$, the condition $(\DRC)$ is not
fulfilled by the family $(\SFP_p)$ in consideration.
\end{example}

By imposing proper additional assumptions, from Theorem \ref{thm:SolvAubin}
one can derive a condition for the local Lipschitz continuity of $\Solv$.
In view of the next result, recall that a set-valued map $F:P\rightrightarrows\X$ is
said to be:
\begin{itemize}
  \item locally compact around $p_0\in P$ if there exist $r>0$ and
a compact set $K\subseteq\X$ such that
$$
  F(p)\subseteq K,\quad\forall p\in\cball{p_0}{r};
$$

  \item closed at $p_0\in P$ if for every $x\not\in F(p_0)$ there exist
  positive $\delta$ and $\epsilon$ such that $F(p)\cap\cball{x}{\epsilon}
  =\varnothing$, for every $p\in\cball{p_0}{\delta}$.
\end{itemize}

\vskip.25cm

\begin{corollary}
Under the hypotheses of Theorem \ref{thm:SolvAubin}, suppose in addition that
either
\begin{itemize}
  \item[$(\tilde{\textrm{v}_1})$] $C:P\rightrightarrows\X$ is locally compact
  around $\bar p$;

or

  \item[$(\tilde{\textrm{v}_2})$] $(\Y,\|\cdot\|)$ is a Banach space, $Q:P\rightrightarrows\Y$ is locally
  compact around $\bar p$, and each $A(p,\cdot)\in\Lin(\X,\Y)$ is invertible, for every
  $p$ in a neighbourhood of $\bar p$;

  \item[$(\textrm{vi})$] $\Solv$ is u.s.c. on $P$.
\end{itemize}
Then, $\Solv$ is locally Lipschitz around $\bar p$, i.e. there exist $\ell>0$ and $\delta>0$
such that
$$
  \haus{\Solv(p_1)}{\Solv(p_2)}\le \ell d(p_1,p_2),\quad\forall
  p_1,\, p_2\in\cball{\bar p}{\delta}.
$$
\end{corollary}

\begin{proof}
Notice that if either $(\tilde{\textrm{v}_1})$ or $(\tilde{\textrm{v}_2})$ is true,
then $\Solv$ turns out to be locally compact around $\bar p$.
Indeed, if $(\tilde{\textrm{v}_1})$ holds, one has immediately
$$
  \Solv(p)\subseteq C(p)\subseteq K,\quad\forall p\in\cball{\bar p}{r_C},
$$
for some $r_C>0$ and some compact subset $K\subseteq\X$.

If $(\tilde{\textrm{v}_2})$ holds, there exist $r_Q>0$ and a compact set
$K\subseteq\Y$ such that
$$
   Q(p)\subseteq K,\quad\forall p\in\cball{\bar p}{r_Q}.
$$
Without loss of generality, it is possible to assume that
$A(p,\cdot)$ is invertible for every $ p\in\cball{\bar p}{r_Q}$.
This implies, in particular, that each $A(p,\cdot)$ is onto and
hence, by virtue of the Banach Open Mapping Theorem, $A(p,\cdot)^{-1}$
is (Lipschitz) continuous on $\Y$. Thus, one obtains
$$
   A(p,\cdot)^{-1}(Q(p))\subseteq A(p,\cdot)^{-1}(K),
   \quad\forall p\in\cball{\bar p}{r_Q},
$$
where the second set in the inclusion is compact, as an image
of a compact set through a continuous map. Consequently, it is possible to write
$$
  \Solv(p)\subseteq A(p,\cdot)^{-1}(K),\quad\forall
  p\in\cball{\bar p}{r_Q},
$$
which shows that $\Solv$ is locally compact around $\bar p$
also in this case.
Moreover, according to what was noticed in Remark \ref{rem:Solvtoppro}, as $\Solv$ is
assumed to be u.s.c., its graph is closed. This fact is known to
entail the property of being closed at $\bar p$.

Then it remains to apply the characterization of
the local Lipschitz continuity for closed and locally compact multifunctions in terms
of the Aubin property, provided by \cite[Theorem 1.42]{Mord06}.
This completes the proof.
\smartqed\qed
\end{proof}


\subsection{A condition for Lipschitz upper semicontinuity}

Lipschitz upper semicontinuity is the only property, among those
appearing in Definition \ref{def:Lipsemcont}, which is not referred to
a point in the graph (see Remark \ref{rem:lscLiplsc}(iii)). Such a feature
requires an adaptation of the approach so far in use. More technically,
the error bound estimate in (\ref{in:serbo}), which is valid only around
the reference point $(\bar p,\bar x)$, is not longer adequate for the present context,
so it must be replaced by a global error bound, as stemming from Lemma
\ref{lem:sovebconvex}. In turn, this change requires a consequent adaptation
of the dual regularity condition.

Fixed $\bar p\in P$ and given $x\in\X$, it is helpful to use again the
notations $\mf_{A,Q}(\bar p,x)=
\dist{A(\bar p,x)}{Q(\bar p)}$, $\mf_C(\bar p,x)=\dist{x}{C(\bar p)}$,
$\bar d_{A,Q}=\dist{A(\bar p,x)}{Q(\bar p)}$ and $\bar d_C=\dist{x}{C(\bar p)}$.
With these notations, let us introduce the following quantities
\begin{eqnarray}   \label{eq:deftauAQinf}
    \tau_{A,Q}(\bar p)=\inf\bigg\{\|x^*\|:\ & x^*\in A(\bar p,\cdot)^*
    \left(\Ncone{A(\bar p,x)}{\cball{Q(\bar p)}{\bar d_{A,Q}})}\cap\Usfer^*\right) \nonumber \\
     & +[\Ncone{x}{\cball{C(\bar p)}{\bar d_C}}\cap\Uball^*],\quad
     x\in [\mf_{A,Q}(\bar p,\cdot)>0]\bigg\}
\end{eqnarray}
and
\begin{eqnarray}   \label{eq:deftauCinf}
   \tau_C(\bar p)=\inf\bigg\{\|x^*\|:\ & x^*\in A(\bar p,\cdot)^*
   \left(\Ncone{A(\bar p,x)}{Q(\bar p)}\cap\Uball^*\right)+[\Ncone{x}{\cball{C(\bar p)}{\bar d_C}}\cap\Usfer^*], \nonumber \\
     &  x\in [\mf_{A,Q}(\bar p,\cdot)=0]\cap[\mf_{C}(\bar p,\cdot)>0]\bigg\}.
\end{eqnarray}
A problem $(\SFP_{\bar p})$ is said to satisfy  the {\it global dual regularity condition} at
$\bar p\in P$ (for short, \GDRC) if
$$
  \tau(\bar p)=\min\{\tau_{A,Q}(\bar p),\,
  \tau_C(\bar p)\}>0.   \leqno(\GDRC)
$$

\vskip.25cm

\begin{theorem}[Lipschitz upper semicontinuity of $\Solv$]    \label{thm:SolvLipusc}
With reference to a family $(\SFP_p)$, let $(\bar p,\bar x)\in\graph\Solv$.
Suppose that:
\begin{itemize}
  \item[(i)]  $C:P\rightrightarrows\X$ is Lipschitz u.s.c. at $\bar p$;

  \item[(ii)] $Q:P\rightrightarrows\Y$ is Lipschitz u.s.c. at $\bar p$;

  \item[(iii)] $A(\cdot,x):P\rightarrow\Y$ is calm at $\bar p$, uniformy in $x\in\X$.

  \item[(iv)] the condition $(\GDRC)$ holds.
\end{itemize}
Then,  $\Solv$ is Lipschitz u.s.c. at $\bar p$ and it holds
\begin{equation}    \label{in:LipuscSolvest}
    \Lipusc{\Solv}{\bar p}\le {\Lipusc{C}{\bar p}
    +\clm{A(\cdot,x)}{\bar p}+\Lipusc{Q}{\bar p}\over \tau(\bar p)}.
\end{equation}
\end{theorem}

\begin{proof}
By applying the sum and the chain rule for the subdifferential
(see formulae (\ref{eq:subdsumr}) and (\ref{eq:subdchainr})), as well as the
representation of the subdifferential of the distance function in term
of normal cone (see formula (\ref{eq:subddist})), one obtains that,
if $(p,x)\in [\mf_{A,Q}(\bar p,\cdot)>0]$, then it holds
$$
  \partial\mf(\bar p,\cdot)(x)\subseteq
  A(\bar p,\cdot)^*\left(\Ncone{A(\bar p,x)}{\cball{Q(\bar p)}{\bar d_{A,Q}})}\cap\Usfer^*\right)
  +[\Ncone{x}{\cball{C(\bar p)}{\bar d_C}}\cap\Uball^*],
$$
whereas, if $(p,x)\in [\mf_{A,Q}(\bar p,\cdot)=0]\cap[\mf_C(\bar p,\cdot)>0]$,
it holds
$$
  \partial\mf(\bar p,\cdot)(x)=
  A(\bar p,\cdot)^*\left(\Ncone{A(\bar p,x)}{A(\bar p)}\cap\Uball^*\right)
  +[\Ncone{x}{\cball{C(\bar p)}{\bar d_C}}\cap\Usfer^*].
$$
Therefore, by virtue of hypothesis (iv), one finds
$$
  \tau_\partial=\inf_{x\in[\mf(\bar p,\cdot)>0]}\dist{\nullv^*}{\partial\mf(\bar p,\cdot)(x)}
  \ge\tau(\bar p)>0,
$$
so the condition in (\ref{def:ebsolvconvex}) is satisfied.
In the light of Lemma \ref{lem:sovebconvex}, as the function $x\mapsto\mf(\bar p,x)$
is convex, real-valued and continuous on $\X$ (remember Remark \ref{rem:Lipcontmf}),
this fact yields the following global error bound estimate
\begin{equation}  \label{in:erboLipusc}
  \dist{x}{\Solv(\bar p)}\le \frac{\mf(\bar p,x)}{\tau(\bar p)},\quad\forall x\in\X.
\end{equation}
Now, according to hypothesis (i), fixed any $\ell_C>\Lipusc{C}{\bar p}$, there exists $\delta_C>0$
such that
$$
  C(p)\subseteq\cball{C(\bar p)}{\ell_C d(p,\bar p)},\quad\forall p\in\cball{\bar p}{\delta_C},
$$
which means that for every $x\in \Solv(p)\subseteq C(p)$
\begin{equation}    \label{in:distxCbarpLipusc}
  \dist{x}{C(\bar p)}\le \ell_C d(p,\bar p),\quad\forall p\in\cball{\bar p}{\delta_C}.
\end{equation}
Analogously, by hypothesis (ii), fixed any $\ell_Q>\Lipusc{Q}{\bar p}$, there exists $\delta_Q>0$
such that
$$
  Q(p)\subseteq\cball{Q(\bar p)}{\ell_Q d(p,\bar p)},\quad\forall p\in\cball{\bar p}{\delta_Q},
$$
which means that for every $y\in Q(p)$
\begin{equation}     \label{in:LipuscQy}
  \dist{y}{Q(\bar p)}\le \ell_Q d(p,\bar p),\quad\forall p\in\cball{\bar p}{\delta_Q}.
\end{equation}
By hypothesis (iii), fixed any $\gamma_A>\clm{A(\cdot,x)}{\bar p}$, there exists $\delta_A>0$
such that
\begin{equation}  \label{in:clmALipusc}
  \|A(p,x)-A(\bar p,x)\|\le\gamma_A d(p,\bar p),  \quad
  \forall p\in\cball{\bar p}{\delta_A},\ \forall  x\in\X.
\end{equation}
Now, define $\delta_\Solv=\min\{\delta_C,\, \delta_Q,\, \delta_A\}$.
From (\ref{in:LipuscQy}) and (\ref{in:clmALipusc}), whenever $x\in\Solv(p)$, so that $A(p,x)\in Q(p)$,
it follows
$$
  \dist{A(\bar p,x)}{Q(\bar p)}\le \|A(\bar p,x)-A(p,x)\|+\dist{A(p,x)}{Q(\bar p)}
  \le (\ell_A+\ell_Q)d(p,\bar p).
$$
Thus, by recalling the inequalities in (\ref{in:distxCbarpLipusc}) and (\ref{in:erboLipusc}),
one obtains that if $x\in\Solv(p)$, then it holds
\begin{eqnarray*}
  \dist{x}{\Solv(\bar p)} &\le& \frac{\dist{A(\bar p,x)}{Q(\bar p)}+\dist{x}{C(\bar p)}}{\tau(\bar p)} \\
   &\le& \frac{\left(\gamma_A+\ell_Q+\ell_C\right)}{\tau(\bar p)}d(p,\bar p),
   \quad\forall p\in\cball{\bar p}{\delta_\Solv}.
\end{eqnarray*}
By arbitrariness of $x\in\Solv(p)$, the above inequality results in
$$
  \Solv(p)\subseteq\cball{\Solv(\bar p)}{\frac{\gamma_A+\ell_Q+\ell_C}{\tau(\bar p)}d(p,\bar p)},
  \quad\forall p\in\cball{\bar p}{\delta_\Solv},
$$
which shows that $\Solv$ is Lipschitz u.s.c. at $\bar p$ with
$$
  \Lipusc{\Solv}{\bar p}\le\frac{\gamma_A+\ell_Q+\ell_C}{\tau(\bar p)}.
$$
From such inequality, by reasoning as already made in estimating the bound
associated to a quantitative semicontinuity property, one can deduce the
inequality in (\ref{in:LipuscSolvest}), thereby completing the proof.
\smartqed\qed
\end{proof}


\section{A special result with perturbations in a normed space}     \label{Sect:4}

Throughout the current section, the metric space $P$
is assumed to have, in particular, the structure of a normed vector space.
To mark the difference of such a specific setting, $(P,d)$ will be henceforth denoted
by $(\Pb,\|\cdot\|)$. Of course, such a specialization results in allowing
less general perturbations.
The next proposition states that the solution map associated to a
family $(\SFP_p)$ with data depending ``convexly" on the perturbation
parameter, inherits the convex dependence.
This result extends to parameterized families of ``split" feasibility problems
a phenomenon, which has been already observed for some of their
specific manifestations (e.g. parametric equality/inequality systems).

\begin{proposition}[Convexity of $\Solv$]    \label{pro:convSolvmap}
With reference to a family $(\SFP_p)$, with $p\in\Pb$, suppose that:
\begin{itemize}
  \item[(i)]  $C:\Pb\rightrightarrows\X$ is convex on $\Pb$;

  \item[(ii)] $Q:\Pb\rightrightarrows\Y$ is convex on $\Pb$;

  \item[(iii)] $A\in\Lin(\Pb\times\X,\Y)$.
\end{itemize}
Then, the multifunction $\Solv:\Pb\rightrightarrows\X$ is convex on $\Pb$.
\end{proposition}

\begin{proof}
Take arbitrary $p_1,\, p_2\in\dom\Solv$, $x_1\in\Solv(p_1)$, $x_2\in\Solv(p_2)$,
and $t\in [0,1]$. This means that $x_1\in C(p_1)\cap A(p_1,\cdot)^{-1}(Q(p_1))$
and $x_2\in C(p_2)\cap A(p_2,\cdot)^{-1}(Q(p_2))$. Since $C$ is convex, from the
above inclusions it follows
\begin{equation}   \label{in:convSolv1}
  t x_1+(1-t)x_2\in tC(p_1)+(1-t)C(p_2)\subseteq C(tp_1+(1-t)p_2).
\end{equation}
Moreover, by the joint linearity of $A$ and the convexity of $Q$, one has
\begin{eqnarray}     \label{in:convSolv2}
  A(t p_1+(1-t)p_2,t x_1+(1-t)x_2) &=& tA(p_1,x_1)+(1-t)A(p_2,x_2)  \\
   &\in & tQ(p_1)+(1-t)Q(p_2)\subseteq Q(t p_1+(1-t)p_2).  \nonumber
\end{eqnarray}
The inclusions in (\ref{in:convSolv1})  and (\ref{in:convSolv2}) result
in
\begin{eqnarray*}
   t x_1+(1-t)x_2 \in  C(tp_1+(1-t)p_2)
   \cap\  A(tp_1+(1-t)p_2,\cdot)^{-1}   \left(Q(t p_1+(1-t)p_2)\right),
\end{eqnarray*}
which entails $t x_1+(1-t)x_2\in\Solv(t p_1+(1-t)p_2)$. By arbitrariness
of $x_1\in\Solv(p_1)$ and $x_2\in\Solv(p_2)$, the last inclusion enables to
conclude that
$$
  t\Solv(p_1)+(1-t)\Solv(p_2)\subseteq\Solv(t p_1+(1-t)p_2),
  \quad\forall p_1,\, p_2\in\dom\Solv.
$$
If at least one among $p_1$ and $p_2$ does not belong to $\dom\Solv$,
the last inclusion becomes trivially true, according to the conventions
$t\cdot\varnothing=\varnothing$ and $\varnothing+\Omega=\varnothing$.
\smartqed\qed
\end{proof}

Elementary counterexamples show that, if one of the hypotheses of
Proposition \ref{pro:convSolvmap} is dropped out, then $\Solv$ may fail to be convex.
As it happens for scalar functions, the combination of convexity and
continuity properties yields Lipschitzian behaviour of multifunction,
provided that a proper qualification conditions is in force.

\begin{theorem}[Aubin property of $\Solv$ under convexity]     \label{thm:Aubproconv}
Given a family $(\SFP_p)$ and $\bar p\in\Pb$, let $(\bar p,\bar x)\in\graph\Solv$.
Suppose that:
\begin{itemize}
  \item[(i)]  $C:\Pb\rightrightarrows\X$ is convex and i.s.c. at $\bar p$;

  \item[(ii)] $Q:\Pb\rightrightarrows\Y$ is convex and i.s.c. at $\bar p$;

  \item[(iii)] $A\in\Lin(\Pb\times\X,\Y)$;

  \item[(iv)] the condition $(\DRC)$ holds.
\end{itemize}
Then,  $\Solv:\Pb\rightrightarrows\X$ has the Aubin property around $(\bar p,\bar x)$.
\end{theorem}

\begin{proof}
Observe that, as noticed in Remark \ref{rem:lscLiplsc}(ii), since
$C$ is i.s.c. at $\bar p$, then it is also l.s.c. at $(\bar p,\bar x)\in\graph C$.
Analogously, since $Q$ is i.s.c. at $\bar p$, then it is also l.s.c.
at $(\bar p,A(\bar p,\bar x))\in\graph Q$.
The hypothesis (iii) entails that $A(\cdot,\bar x)$ is continuous at $\bar p$.
Thus, under the condition $(\DRC)$, it is possible to apply Theorem \ref{thm:Solverbo},
which ensures that, for some positive $\zeta$ and $\eta$,
\begin{equation}    \label{neq:Solvnonempt}
    \Solv(p)\ne\varnothing,\quad\forall p\in\cball{\bar p}{\zeta},
\end{equation}
and the error bound estimate in (\ref{in:serbo}) holds.
Notice that (\ref{neq:Solvnonempt}) implies that the function $(p,x)\mapsto\dist{x}{\Solv(p)}$
is real-valued around $(\bar p,\bar x)$. Moreover, by the convexity of $C$ and $Q$,
being $A\in\Lin(\Pb\times\X,\Y)$,
according to Proposition \ref{pro:convSolvmap}, such function is convex on $\Pb\times\X$.
On the other hand, by recalling Remark \ref{rem:lscLiplsc}(ii), since $C$ is i.s.c.
at $\bar p$, one can say that the function $(p,x)\mapsto\dist{x}{C(p)}$ is
u.s.c. at $(\bar p,\bar x)$.
Similarly, since $Q$ is i.s.c. at $\bar p$, one can say that the function $(p,y)\mapsto
\dist{y}{Q(p)}$ is u.s.c. at $(\bar p,y)$, for every $y\in Q(\bar p)$. In turn,
by continuity of $A$, this implies that the function $(p,x)\mapsto\dist{A(p,x)}{Q(p)}$
is u.s.c. at $(\bar p,\bar x)$. Consequently, as a sum of two functions u.s.c. at $(\bar p,\bar x)$,
$\mf$ is u.s.c. at the same point.
Thus, as it is
$$
  \dist{\bar x}{\Solv(\bar p)}=0\le\dist{x}{\Solv(p)}\le\tau(\delta)^{-1}\mf(p,x),\quad\forall (p,x)
  \in\cball{\bar p}{\zeta}\times\cball{\bar x}{\eta},
$$
the function $(p,x)\mapsto\dist{x}{\Solv(p)}$ turns out to be continuous
at the point $(\bar p,\bar x)$ interior to its domain.
By convexity, this fact is known to imply the Lipschitz continuity of that
function around $(\bar p,\bar x)$ (see, for instance, \cite[Theorem 2.149]{MorNam22}).
In the light of the scalarization criterion formulated in
Proposition \ref{pro:scacriAubcont}, what has been obtained
is sufficient for $\Solv$ to have the Aubin property around $(\bar p,\bar x)$.
This completes the proof.
\smartqed\qed
\end{proof}

\begin{remark}
(i) Whenever $(\Pb,\|\cdot\|)$ is a Banach space, a shortcut for
the proof of Theorem \ref{thm:Aubproconv} can be devised, which relies on the
Robinson-Ursescu Theorem (see \cite[Theorem 5B.4]{DonRoc14}).
Indeed, the nonemptiness in (\ref{neq:Solvnonempt}) says that $\bar p
\in\inte\dom\Solv=\inte\Solv^{-1}(\X)$, namely $\bar p$ is
an interior point of the range of $\Solv^{-1}$.  In consideration of the convexity
of $\Solv$, and hence of $\Solv^{-1}$, by the Robinson-Ursescu Theorem this fact is equivalent to
the metric regularity of inverse multifunction $\Solv^{-1}$ around $(\bar x,\bar p)$,
which is known to be an equivalent reformulation of the Aubin property of $\Solv$
around $(\bar p,\bar x)$.

(ii) Whenever  $\Pb$ and $\X$ are finite-dimensional Euclidean spaces,
the hypotheses (i) and (ii) in Theorem \ref{thm:Aubproconv} can be
weakened, by replacing the inner semicontinuity of $C$ and $Q$ with their mere
lower semicontinuity at $(\bar p,\bar x)$ and $(\bar p, A(\bar p,\bar x))$,
respectively. Indeed, such weaker assumptions still enable to apply
Theorem \ref{thm:Solverbo}, so as to get that $(\bar p,\bar x)$ belongs
to the interior of the domain of the function $(p,x)\mapsto\dist{x}{\Solv(p)}$.
As convex functions are locally Lipschitz on the interior of finite-dimensional
domains (see \cite[Corollary 2.152]{MorNam22}), in such a circumstance one can
directly apply Proposition \ref{pro:scacriAubcont}.
\end{remark}

\vskip.5cm


\section{Conclusions}

The findings emerging from the exploration of well-known Lipschitzian properties
for those maps implicitly defined by a parameterized family
of ``split" feasibility problems tell us that the validity of each of these
properties for the problem data multifunctions can be inherited by the solution
map, provided that an adequate qualification condition is assumed to hold.
According to the proposed approach, such a qualification condition can
be read as a subdifferential slope condition, tying the behaviour
of all the problem data. In fact, it yields local solvability and
metric estimates, as if stemming from subtransversality assumptions.
The findings of the present study can serve as a starting point for investigating
further aspects of the perturbation analysis of ``split" feasibility problems
and as a solid base for elaborating approximated methods for their resolution.

\vskip.5cm


\begin{acknowledgements}
The current version of the paper has benefited from helpful discussions with Professor
Y. Censor from the University of Haifa, to whom the author is grateful.
\end{acknowledgements}



\end{document}